\documentclass[11pt,reqno]{amsart}

\usepackage{amssymb,latexsym,verbatim,pdfsync,color,mathrsfs,graphicx,caption} 

\newcommand{\bbC}{{\mathbb C}}

\newcommand{\bbR}{{\mathbb R}}
\newcommand{\bbT}{{\mathbb T}}
\newcommand{\bbX}{{\mathbb X}} 
\newcommand{\bbZ}{{\mathbb Z}} 

\newcommand{\p}{\partial}

\DeclareMathOperator{\Log}{Log}
\DeclareMathOperator{\sgn}{sgn}

\def\re{\operatorname{Re}}
\def\im{\operatorname{Im}}

\def\ts{\textstyle}

\def\bC{{\mathbf C}}

\def\cC{{\mathcal C}}
\def\cD{{\mathcal D}}

\def\cF{{\mathcal F}}

\def\cH{{\mathcal H}}

\def\cM{{\mathcal M}}

\def\cP{{\mathcal P}}

\def\cS{{\mathcal S}}

\def\cU{{\mathcal U}}

\def\cW{{\mathcal W}}

\def\sH{{\mathscr H}}

\def\sP{{\mathscr P}}

\def\sT{{\mathscr T}}

\def\Re{\operatorname{Re}}
\def\Im{\operatorname{Im}}

\def\eps{\varepsilon}
\def\vp{\varphi}

\def\p{\partial}

\def\ms{\medskip}

\def\pt{\textstyle{\frac{n}{2}}}

\def\pt{{\textstyle \frac{\pi}{2} }}

\def\Hol{\operatorname{Hol}}

\def\CB{\color{black} }

\newtheorem{thm}{Theorem}[section]
\newtheorem{prop}[thm]{Proposition}
\newtheorem{cor}[thm]{Corollary}
\newtheorem{lem}[thm]{Lemma}
\newtheorem{defn}[thm]{Definition}
\newtheorem{remark}[thm]{Remark}

\newtheorem{theorem}{Theorem}

\DeclareFontFamily{U}{mathx}{\hyphenchar\font45}
\DeclareFontShape{U}{mathx}{m}{n}{
      <5> <6> <7> <8> <9> <10>
      <10.95> <12> <14.4> <17.28> <20.74> <24.88>
      mathx10
      }{}
\DeclareSymbolFont{mathx}{U}{mathx}{m}{n}
\DeclareFontSubstitution{U}{mathx}{m}{n}
\DeclareMathAccent{\widecheck}{0}{mathx}{"71}
\DeclareMathAccent{\wideparen}{0}{mathx}{"75}

\begin{document}

\title[Szeg\H o projection on the distinguished boundary]{Sharp estimates for the Szeg\H o 
  projection on the distinguished boundary of model worm domains}

\author[A. Monguzzi]{Alessandro Monguzzi}
\author[M. M. Peloso]{Marco M. Peloso}
\address{Dipartimento di Matematica ``F. Enriques''\\
Universit\`a degli Studi di Milano\\
Via C. Saldini 50\\
I-20133 Milano}

\email{alessandro.monguzzi@unimi.it}
\email{marco.peloso@unimi.it}
\thanks{Both authors supported in part by the 2010-11 PRIN grant
  \emph{Real and Complex Manifolds: Geometry, Topology and Harmonic Analysis}  
  of the Italian Ministry of Education (MIUR)}


\begin{abstract}
In this paper we study the regularity of the  Szeg\H o projection on
Lebesgue and Sobolev spaces on the distinguished 
boundary of the unbounded model
worm domain $D_\beta$.  

We denote by $d_b(D_\beta)$
the distinguished boundary of $D_\beta$ and define 
the corresponding Hardy space  $\sH^2(D_\beta)$. This can be
identified with a closed subspace of $L^2(d_b(D_\beta),d\sigma)$, that 
we denote by $\sH^2(d_b(D_\beta))$,
where $d\sigma$ is the naturally induced measure on $d_b(D_\beta)$. 

The orthogonal Hilbert space projection $\sP: L^2(d_b(D_\beta),d\sigma)\to 
\sH^2(d_b(D_\beta))$ is called the Szeg\H o projection on the distinguished 
boundary.

We prove that $\sP$, initially defined on the dense subspace $L^2(d_b(
D_\beta),d\sigma)\cap L^p(d_b(D_\beta), d\sigma)$ extends to a bounded
operator 
$\sP: L^p(d_b(D_\beta), d\sigma)\to L^p(d_b(D_\beta),
d\sigma)$ if and only if
$\textstyle{\frac{2}{1+\nu_\beta}}<p<\textstyle{\frac{2}{1-\nu_\beta}}$
where
$\nu_\beta=\textstyle{\frac{\pi}{2\beta-\pi}},\beta>\pi$. Furthermore,
we also prove that $\sP$ defines a bounded operator $\sP:
W^{s,2}(d_b(D_\beta),d\sigma)\to W^{s,2}(d_b(D_\beta), d\sigma)$ if and
only if $0\leq s<\textstyle{\frac{\nu_\beta}{2}}$ where $W^{s.2}(d_b(
D_\beta), d\sigma)$ denotes the Sobolev space of order $s$ and
underlying $L^2$-norm. 

Finally, we prove a necessary condition for the boundedness of $\sP$
on $W^{s,p}(d_b(D_\beta), d\sigma)$, $p\in(1,\infty)$, the Sobolev space
of order $s$ and underlying $L^p$-norm.   
\end{abstract}  

\maketitle

\ms

\section{Introduction and statement of the main results}
\label{sec1}
In this paper we consider the Hardy spaces on the distinguished
boundary of the model worm domain 
\begin{equation*}
D_\beta =
\Big\{ (z_1,z_2)\in\bbC^2:\, \Re\big(z_1e^{-i\log|z_2|^2}\big)>0,\ 
\big| \log|z_2|^2\big|<\beta-\pt \Big\}, \quad \beta>\pi,
\end{equation*}
and study the continuity of the associated Szeg\H{o} projection
operator on the Lebesgue and Sobolev spaces. 

The domain $D_\beta$ is unbounded, pseudoconvex, with Lipschitz
boundary, and was instrumental in proving that the Bergman
projection does not preserve Sobolev spaces of sufficiently high order
on the smooth pseudoconvex worm domain $\Omega_\beta$ introduced by
K. Diederich and J.E. Forn\ae ss \cite{MR0430315}. This result, due to
D. Barrett \cite{Ba-Acta}, was a major breakthrough since it suggested
that some long-standing conjectures about the geometry of smooth
pseudoconvex domains and the regularity of the associated Bergman
projection were actually false. Indeed, few years later, M. Christ
\cite{MR1370592} proved that
the Bergman projection on the Diederich--Forn\ae ss worm
domain does not preserve the space $C^\infty(\overline{\Omega_\beta})$
-- that is, Condition $R$ fails on $\Omega_\beta$.
We refer the reader to \cite{MR2393268} for a detailed
account on the subject. 

The question of the regularity of the Bergman projection on worm
domains and related questions
has been considered by various authors, and
here we mention in particular
\cite{KP-Houston, MR2904008,  MR3424478, 2015arXiv150906383K, KPS, CS}.   
\medskip

On worms domains, of course
it is  of great interest also to  study the (ir-)regularity of  the
boundary analogue of the Bergman 
projection, that is, the Szeg\H{o} projection.   If $\Omega=\{ z:\rho(z)<0\}$ is a
smoothly bounded domain in $\bbC^n$, the Hardy space $H^2(\Omega)$ is
defined as 
$$
H^2(\Omega) =\big\{ f\in\Hol(\Omega): \sup_{\eps>0}
\int_{b\Omega_\eps} |f|^2 d\sigma_\eps <\infty
\big\} \,,
$$
where $\Omega_\eps=\{ z:\rho(z)<-\eps\}$ and $d\sigma_\eps$ is the
induced surface measure on $b\Omega_\eps$.  Here, and in what follows,
we denote by $b\cD$ the topological boundary of a given domain $\cD$.
Then, $H^2(\Omega)$ can be 
identified with a closed subspace of $L^2(b\Omega,d\sigma)$, that we
denote by $H^2(b\Omega)$.  The
Szeg\H o projection is the orthogonal projection 
$$
P_\Omega : L^2(b\Omega,d\sigma) \to H^2(b\Omega) \,;
$$
see \cite{St2}. 
Mapping properties
of the Szeg\H o projection on other function spaces
have been studied for various classes of smooth bounded
domains. In the case of strictly pseudoconvex domains
\cite{MR0450623},  domains of finite type in $\bbC^2$
\cite{MR979602}, convex domains of finite type in
$\bbC^n$ \cite{MR1452048} the Szeg\H o projection $
P_\Omega$ turns out to be bounded on the Lebesgue--Sobolev spaces
$W^{s,p}(b\Omega)$ for $1<p<\infty$ and
$s\ge0$.  When $\Omega$ is Reinhardt domain \cite{MR773403,MR835396},
a domain with partially transverse symmetries \cite{MR971689}, a
pseudoconvex domain satisfying Catlin's property $(\cP)$ \cite{
  MR871667}, a 
complete Hartogs domain in $\bbC^2$ \cite{ MR999739}, or a domain with
a plurisubharmonic defining function on the boundary \cite{MR1133741},
then  
the Szeg\H{o} projection is exactly regular, that is, $P_\Omega$ is a
bounded operator $P_{\Omega}: W^{s,2}(b\Omega)\to W^{s,2}(b\Omega)$
for every $s\geq0$.  We also mention that, if $\Omega$ is bounded,
$C^2$ and strongly pseudoconvex in $\bbC^n$, the Szeg\H{o} projection
$P_\Omega$ again extends to bounded operator on $L^p(b\Omega)$ for
$1<p<\infty$, 
\cite{2015arXiv150603748L,2015arXiv150603965L}. 

There are examples of domains $\Omega$  on which the Szeg\H{o} projection
$P_\Omega$ is less regular.  L. Lanzani and E.M. Stein described the
(ir-)regularity of $P_\Omega$ on Lebesgue spaces in the case of planar
simply connected domains, \cite[Thm. 2.1]{MR2030575}.  In particular
they showed that if $\Omega$ has Lipschitz boundary, then
$P_\Omega:L^p(b\Omega)\to L^p(b\Omega)$ if and only if
$p'_\Omega<p<p_\Omega$, where $p_\Omega$ depends only on the Lipschitz
constant of $b\Omega$. 
More recently, S. Munasinghe and Y.E. Zeytuncu provided an example of
a piecewise smooth, bounded pseudoconvex domain in $\bbC^2$ on which
the Szeg\H{o} projection
$P_\Omega$ is unbounded on
$L^p(b\Omega)$ for every $p\neq 2$ \cite{MR3355787}. The same result on tube
domains over irriducible self-dual cones of rank greater than 1 has
been known for a number of years, \cite{BeBo}.\ms

Clearly the  Szeg\H{o} projection depends
on the choice of the measure on the boundary.   For instance, instead
of taking the induced surface measure $d\sigma$, one can consider the
Fefferman surface measure, see
 \cite{MR3145917}, or any surface measure of the form $\omega
 d\sigma$, with $\omega$ continuous and positive,
 \cite{2015arXiv150603965L}.  On non-smooth domains, such as the
 polydisk, it is also of interest, and perhaps more natural, to study
 Hardy spaces defined by integration over the so-called {\em
   distinguished boundary}. For a general domain
 $\Omega\subseteq\bC^n$, the
 distinguished boundary $d_b(\Omega)$ is the set
\begin{equation}\label{dist-bndry-def1}
d_b(\Omega)=\Big\{\zeta\in b\Omega:  \sup_{z\in b\Omega} |f(z)| \le
\sup_{\zeta\in d_b(\Omega)} |f(\zeta)| \ \text{for all\ } f\in H^\infty(\Omega)\Big\} 
\end{equation}
  where $H^\infty(\Omega)$ denotes the spaces of bounded holomorphic
  functions on $\Omega$. 

\medskip

In this paper we obtain sharp results concerning the regularity  
of the Szeg\H{o} projection $\sP=\sP_{D_\beta}$ for the Hardy spaces  on the
distinguished boundary of the model worm domain
$D_\beta$, that is
\begin{equation}\label{dist-bndry-def}
d_b(D_\beta)=\Big\{(z_1,z_2)\in\bbC^2: \big|\arg\,
  z_1-\log|z_2|^2\big|=\pt,\,  \big|\log|z_2|^2\big|=\beta-\pt\Big\}  \,.
\end{equation}
Our final goal is to study the (ir-)regularity of the
Szeg\H{o} projection of the smooth worm $\Omega_\beta$ and the present
work is a step in this direction. As is the Bergman setting, it is
reasonable to expect that 
the peculiar geometry of the worm domain affects the regularity of the
Szeg\H o projection as well. The results in this paper support this
expectancy.  
\medskip

In shifting from the Bergman to the Szeg\H o setting new difficulties
arise and Barrett's arguments cannot be trivially adapted to transfer
information from the model worm $D_\beta$ to the smooth worm
$\Omega_\beta$ \cite{MR3491881}. Barrett proved the irregularity of
the Bergman projection $B_{D_\beta}$ on $W^{s,2}(\Omega_\beta)$
if $s\ge \pi/(2\beta-\pi)$
by means of the well-known
transformation rule for the Bergman projection and by
studying the Bergman projection of a biholomorphic copy
of $D_\beta$, classically denoted by $D'_\beta$. 
Subsequently, if $\Omega_{\beta,\lambda}$, $\lambda>0$,
denotes an appropriate dilation of $\Omega_\beta$, Barrett proved that
$B_{\Omega_{\beta,\lambda}}\to B_{D_\beta}$ as $\lambda\to\infty$ in a
suitable sense.  Thus, if
$B_{\Omega_\beta}$ were bounded on $W^{s,2}(\Omega_\beta)$, 
then also  $B_{D_\beta}$ would be bounded on $W^{s,2}(D_\beta)$. Hence, by
contradiction, it follows that $B_{\Omega_\beta}$ cannot be
bounded on $W^{s,2}(\Omega_\beta)$ if $s\ge \pi/(2\beta-\pi)$.
\medskip

In the Szeg\H{o} setting, in general, there is no transformation rule
for the Szeg\H{o} projection under biholomorphic mappings. 
Nonetheless, we are able to prove a transformation rule for the
projections $\sP$ and $\sP'=\sP_{D'_\beta}$, 
 see Section
\ref{xxx}. 
The mapping properties of the Szeg\H{o} projection $\sP'$ on the
distinguished boundary of $D'_\beta$ were studied by the first author
in \cite{M, MR3506304}. 
\ms

We now describe our results in greater details.
For $(t,s)\in (0,\pt)\times[0,\beta-\pt)$ consider the 
domain
$$
D_{t,s}=\left\{(z_1,z_2)\in\bbC^2: \big|\arg\, z_1-\log|z_2|^2\big|<t,
  \big|\log|z_2|^2\big|<s\right\}. 
$$
Then, the collection $\{D_{t,s}\}_{t,s}$ is a family of
approximating domains for $D_{\beta}$ and the distinguished boundary
of each of these domains is  
\begin{equation*}
d_b(D_{t,s})=\left\{(z_1,z_2)\in\bbC^2: \big|\arg\,
  z_1-\log|z_2|^2\big|=t, \big|\log|z_2|^2\big|=s\right\}. 
\end{equation*}
Consequently,
for $1\le p<\infty$, we define the Hardy space $\sH^p(D_\beta)$
as
\begin{equation*}
\sH^p(D_\beta):=\bigg\{ f\in\Hol(D_\beta):
\, \|f\|^p_{\sH^p(D_\beta)}=\sup_{(t,s)\in(0,\pt)\times[0,\beta-\pt)}
\|f\|^p_{L^p(d_b(D_{t,s}))}<\infty\bigg\} \,,
\end{equation*}
where, denoting by  $d\sigma_{t,s} $ the induced measure on $d_b(
D_{t,s})$,
\begin{align}\label{GrowthD}
 & \|f\|^p_{L^p(d_b(D_{t,s}))}
= \int_{d_b(D_{t,s})} |f|^p\,
 d\sigma_{t,s} \notag \\
& =  \int_0^{\infty}\int_0^{2\pi}|f\big(re^{i(s+t)}, e^{\frac{s}{2}}
e^{i\theta}\big)|^p\ e^{\frac{s}{2}} d\theta dr
+\int_0^{\infty}\int_0^{2\pi} |f\big(r e^{i(s-t)},
 e^{\frac{s}{2}}e^{ i\theta}\big)|^p \ e^{\frac{s}{2}} d\theta dr
  \\ 
 & \quad +\int_0^{\infty}\int_0^{2\pi} |f\big(re^{-i(s+t)},
 e^{-\frac{s}{2}}
e^{i\theta}\big)|^p\ e^{-\frac{s}{2}} d\theta dr
+\int_0^{\infty}\int_0^{2\pi} |f\big(r e^{-i(s-t)},
  e^{-\frac{s}{2}}e^{ i\theta}\big)|^p \ e^{-\frac{s}{2}}d\theta dr \,.\notag
\end{align}
  In Section \ref{Har-space-def-section} we will discuss the above
definition and show that it is a very natural one.  

Standard basic facts of Hardy space theory give that
any $f\in \sH^p$ admits a boundary value function,
that we still denote by $f$, defined on 
$d_b(D_\beta)$ that is $p$-integrable w.r.t. $d\sigma$, the
induced surface measure  
on $d_b(D_\beta)$.  Moreover, we
have the equality
$$
\|f\|_{\sH^p(D_\beta)}^p 
= \int_{d_b(D_\beta)} |f|^p\, d\sigma\,.
$$
Then, we can identify $\sH^p$ with a subspace of
$L^p(d_b(D_\beta),d\sigma)= L^p(d_b(D_\beta))$, that is closed and that we denote by
$\sH^p(d_b(D_\beta))$.  
The Szeg\H{o} projection $\sP$ is the the orthogonal projection
of $L^2(d_b(D_\beta))$ onto $\sH^2(d_b(D_\beta))$ 
$$
\sP: L^2(d_b(D_\beta)) \to \sH^2(d_b(D_\beta)) \,.
$$
The operator $\sP$ admits an integral representation and we call its
integral kernel the Szeg\H{o} kernel. It turns out that in this case
the operator $\sP$ can be written as a combination of Mellin--Fourier
multiplier operators (see Section \ref{M-F}). Our main results are the following.

\begin{theorem}\label{t:LpBounds}
Let be $\nu_\beta=\frac{\pi}{2\beta-\pi}$. Then, the Szeg\H{o}
projection $\sP$, initially defined on the dense subspace $L^p(\p
D_\beta)\cap L^2(d_b(D_\beta))$,  extends to a bounded operator   
$$
\sP:L^p(d_b(D_\beta))\to L^p(d_b(D_\beta))
$$
if and only if $\frac{2}{1+\nu_\beta}<p<\frac{2}{1-\nu_\beta}$.
\end{theorem}
\begin{theorem}\label{SobolevL2}
Let be $\nu_\beta=\frac{\pi}{2\beta-\pi}$. Then, the Szeg\H{o}
projection $\sP$ defines a bounded operator  
$$
\sP: W^{s,2}(d_b(D_\beta))\to W^{s,2}(d_b(D_\beta))
$$
if and only if $0\leq s<\frac{\nu_\beta}{2}$.
\end{theorem}
Our results thus completely describes the mapping properties of the
operator $\sP$ with respect to the Lebesgue and Sobolev--Hilbert
norms. If we consider Sobolev norms with $p\neq 2$ we do not have a complete
characterization of the mapping properties of $\sP$, but we have a
partial result.
\begin{theorem}\label{SobolevLp}
 Let be  $\nu_\beta=\frac{\pi}{2\beta-\pi}$, $s>0$ and
 $p\in(1,\infty)$. If the operator $\sP$, initially defined on the
 dense subspace $W^{s,p}(d_b(D_\beta))\cap L^2(d_b(D_\beta))$, extends to
 a bounded operator $\sP: W^{s,p}(d_b(D_\beta))\to W^{s,p}(d_b(
 D_\beta))$, then 
$$
-\frac{\nu_\beta}{2}\leq s+\frac{1}{2}-\frac{1}{p}\leq\frac{\nu_\beta}{2}.
$$
Assuming $p\geq2$ we obtain the stronger condition
\begin{equation*}
 0\leq s+\frac12-\frac1p<\frac{\nu_\beta}{2}.
\end{equation*}
\end{theorem}

In this work, $W^{s,p}(d_b(D_\beta))$ for a non-integer $s$ denotes
the classical fractional Sobolev space defined via the Fourier
transform and the Bessel potentials. We refer the reader to Section
\ref{SobolevL2-irregularity} for the precise definition. 

\ms

The paper is organized as follows. In Section \ref{M-F} we introduce
Mellin--Fourier multiplier operators and we provide a sufficient
condition for their $L^p$ continuity. In Section
\ref{Har-space-def-section} we study the space  $\sH^2(D_\beta)$, we
show that the Szeg\H{o} projection $\sP$ is given by a sum of
Mellin--Fourier multiplier operators and we prove the sufficient
condition of Theorem \ref{t:LpBounds}. In Section \ref{irregularity}
we conclude the proof of Theorem \ref{t:LpBounds} and we prove the
Sobolev irregularity, whereas in Section \ref{Sobolev-regularity} we
prove the Sobolev regularity. In Section \ref{xxx} we prove 
the transformation rule for $\sP$ and $\sP'$.  We conclude the paper 
with some
final remarks.

\bigskip 

\section{Mellin--Fourier multiplier operators} \label{M-F}
\ms

In this section we introduce a class of operators, which we call
Mellin--Fourier multiplier operators, and provide a sufficient
condition for their $L^p$ continuity, $p\in(1,\infty)$. A similar
class of operators was studied by Rooney \cite{MR860095}.  
Incidentally, we believe that this class of operators is of its own interest. \ms

We are going to work both on $\bbR^2$ and on $\bbR\times\bbT$, where
$\bbT=\bbR/\bbZ$ is the torus.  Thus, we set $\bbX$ to denote either
$\bbR$ or $\bbT$,  and, accordingly, 
$\widehat{\bbX}= \bbR$, or $\bbZ$, respectively.   
  This allows us to unify the presentation and should
cause no confusion. 

We denote by $\cF$ the Fourier transform on $\bbR\times\bbX$, given by
$$
\cF f(\xi_1,\xi_2) =
\int_{\bbR\times\bbX} f(x_1,x_2) e^{- i(x_1\xi_+x_2\xi_2)}\, dx_1dx_2
$$
when $f$ is absolutely integrable, and by
  $T_m$ the Fourier
multiplier operator
$$
T_m(f) = \cF^{-1} \big(m \widehat f\,\big) \,
$$
when $m$ is a bounded measurable function on
$\bbR\times\widehat{\bbX}$, 
where we also write $\cF
f=\widehat f$.   
We say that a bounded function $m$ on $\bbR\times\bbX$ is a {\em
  bounded Fourier multiplier} on $L^p(\bbR\times\bbX)$ if 
$T_m: L^p(\bbR\times\bbX)\to L^p(\bbR\times\bbX)$ is bounded.\ms 
\CB

For  a function  $\vp\in C^\infty_c (0,\infty)$ we define the operator $\cC_p$
\begin{equation*}
 \cC_p \vp(x) = e^{\frac xp} \vp(e^x) \,.
\end{equation*}
Clearly $\cC_p$ extends to an isometry of  $L^p((0,+\infty))$  onto
$L^p(\bbR)$.  With an abuse of notation, for  a fixed $\alpha\in\bbR$,
we also denote by $\cC_p$
the operator defined on functions in $C^\infty_c
((0,+\infty)e^{i\alpha} \times
\bbX)$ and acting on the first variable only, that is,
$$
\cC_p \vp(x+i\alpha,y) = e^{\frac 1p(x+i\alpha)} \vp(e^{x+i\alpha},y) \,.
$$
For $a,b\in\bbR$, with $0<a<b<1$, we denote by $S_{a,b}$ the vertical
strip in the complex plane
\begin{equation*}
S_{a,b}=\big\{z\in\bbC:\,  a<\Re z<b \big\}\,.
\end{equation*}
Given a bounded measurable function $m$ defined on
$S_{a,b}\times\bbX$, when $a<\frac1p<b$
we write
\begin{equation*}
m_p(\xi_1,\xi_2)=m(\textstyle{\frac1p}- i\xi_1,\xi_2) \,. 
\end{equation*}

Finally,
 we define an operator acting on functions defined on 
$(0,+\infty)\times \bbX$ 
as
\begin{equation}\label{def-MF-multiplier}
\sT_{m,p} =\cC_p^{-1} T_{m_p} \cC_p \,.
\end{equation}
We call such an operator a  Mellin--Fourier multiplier operator, the
reason for which will soon be  clear.

\begin{thm}\label{rooney-thm}
With the above notation, let $m:S_{a,b}\times\bbX\to \bbC$ be
continuous and such that 
\begin{itemize}
\item[(i)]
$m(\cdot,\xi_2)\in\Hol(S_{a,b})$ and bounded in every closed substrip
of $S_{a,b}$,
for every $\xi_2\in\bbX$ fixed;\smallskip
\item[(ii)] for every $q$ such that $a<\frac1q<b$,
$m_q$ is a
  bounded Fourier multiplier on $L^q(\bbR\times\bbX)$.
\end{itemize}
Then, for $a<\frac1p<b$,  $\sT_{m,p}=\sT_m$ is independent  
of $p$ and  
$$
\sT_m: L^p((0,+\infty)\times\bbX)
\to L^p( (0,+\infty)\times\bbX) 
$$
is bounded.
\end{thm}

\proof
The fact that $\sT_{m,p}: L^p((0,+\infty)\times\bbX)
\to L^p( (0,+\infty)\times\bbX) $
is bounded is clear since for $\vp\in C^\infty_c ((0,+\infty)\times
\bbX)$
\begin{align*}
\| \sT_{m,p}(\vp)\|_{ L^p((0,+\infty)\times\bbX)}
& = \| T_{m_p}  \cC_p(\vp)\|_{ L^p(\bbR\times\bbX)} \\
& \le
\|m_p\|_{\cM_p(\bbR\times\bbX)} \| \cC_p(\vp)\|_{ L^p(\bbR\times\bbX)}
\\
& = \|m_p\|_{\cM_p(\bbR\times\bbX)} \| \vp
\|_{L^p((0,+\infty)\times\bbX)} \,.
\end{align*}
Here we denote by $\|m_p\|_{\cM_p(\bbR\times\bbX)} $ the operator norm
of the bounded Fourier multiplier $m$ on $L^p(\bbR\times\bbX)$.

Thus, it suffices to show that $\sT_{m,p}$ is independent of $p$, when 
$a<\frac1p<b$.  We show the argument in the case $\bbX=\bbR$ since the
case 
$\bbX=\bbT$ is identical, one only has to replace integration on
$\bbR$ w.r.t. to $d\xi_2$ by a summation over $\bbZ$. 

Setting $g=\cC_p(\vp)$ we observe that, for
$(t,x_2)\in (0,+\infty)\times\bbX)$, by Fubini's theorem,  
\begin{align*}
&\cC_p^{-1}  T_{m_p}\cC_p (\vp)  (t,x_2)\\
& = t^{-\frac1p}    T_{m_p} g(\log t,x_2) =  t^{-\frac1p} \cF^{-1} \big( m_p
\widehat g\big) (\log t,x_2) \\
& = \frac{1}{(2\pi)^2}  t^{-\frac1p} \int_{\bbR\times\bbR} 
e^{ i(\xi_1   \log t+\xi_2 x_2)} m(\textstyle{\frac1p}  - i\xi_1,\xi_2) \widehat g
(\xi_1,\xi_2) \, d\xi_1 d\xi_2 \\
& = \frac{1}{(2\pi)^2}\int_\bbR \int_\bbR t^{-\frac1p + i \xi_1}
 m({\textstyle{\frac1p}}- i\xi_1,\xi_2) \widehat g(\xi_1,\xi_2) \, d\xi_1 
\, e^{ i\xi_2 x_2}\, d\xi_2 \\
& = \frac{1}{(2\pi)^2}\int_\bbR \int_\bbR t^{-\frac1p + i \xi_1}
 m({\textstyle{\frac1p}}- i\xi_1,\xi_2) 
\int_{\bbR^2} e^{- i(y_1\xi_1+y_2\xi_2)} g(y_1,y_2)\, dy_1dy_2
\, d\xi_1 
\, e^{ i\xi_2 x_2}\, d\xi_2 \\
& =\frac{1}{(2\pi)^2} \int_\bbR \int_\bbR t^{-\frac1p + i \xi_1}
 m({\textstyle{\frac1p}}- i\xi_1,\xi_2) 
\int_{\bbR^2} e^{- i(y_1\xi_1+y_2\xi_2)} e^{\frac{y_1}{p}} \vp (e^{y_1},y_2)\, dy_1dy_2
\, d\xi_1 
\, e^{ i\xi_2 x_2}\, d\xi_2 \\
& = \frac{1}{(2\pi)^2}\int_\bbR \int_\bbR t^{-\frac1p + i \xi_1}
 m({\textstyle{\frac1p}}- i\xi_1,\xi_2) 
\int_\bbR \int_0^{+\infty}  
\tau^{\frac1p - i \xi_1-1} \vp (\tau,y_2)\, d\tau  \,  e^{- i y_2\xi_2}dy_2
\, d\xi_1 
\, e^{ i\xi_2 x_2}\, d\xi_2 \,.
\end{align*}
Hence, if we denote by $M_1$ the Mellin transform  in the
first variable of a sufficiently
regular funciton $\psi$ defined on $(0,+\infty)\times\bbR$, that is, 
$$
M_1 \vp(z,\xi_2) = \int_0^{+\infty} t^{z-1} \vp(t,\xi_2)\, dt\,,
$$
and by $\cF_2$ the Fourier transform in the second variable, we see
that 
\begin{multline}
\cC_p^{-1}  T_{m_p}\cC_p (\vp)  (t,x_2)
 \\
= \frac{1}{(2\pi)^2} \int_\bbR \int_\bbR t^{-\frac1p + i \xi_1}
m({\textstyle{\frac1p}}- i\xi_1 ,\xi_2) 
(M_1 \cF_2 \vp)({\textstyle{\frac1p}} - i\xi_1 ,\xi_2) 
\, d\xi_1 
\, e^{ i\xi_2 x_2}\, d\xi_2 \,. \label{Tm-above}
\end{multline}

Now, for every $\xi_2$ fixed, the function $M_1 \big(\cF_2\vp(\cdot,\xi_2)\big)$ is
holomorphic in the right half-plane and bounded in every closed strip 
$\overline{S_{\delta,R}}$, with $0<\delta<R<\infty$, and integrable on
every vertical line in the  right half-plane.  Thus, for every $\xi_2$ fixed,
$m(\cdot,\xi_2)M_1 \big(\cF_2\vp(\cdot,\xi_2)\big)$ is
holomorphic  and bounded in every closed strip contained in
$S_{a,b}$, and integrable on
every vertical line in $S_{a,b}$.   A standard application of Cauchy's
theorem, as in the inversion of the Mellin transform, shows that 
\begin{equation*}
\int_{c+i\bbR} t^{-z} m(z,\xi_2) M_1 \big(\cF_2\vp
(\cdot,\xi_2)\big)(z)\, dz 
\end{equation*}
is independent of $c\in (a,b)$.  Thus, from \eqref{Tm-above} we have
\begin{align*}
\cC_p^{-1}  T_{m_p}\cC_p (\vp)  (t,x_2)
& = \frac{1}{(2\pi)^2} \int_\bbR \bigg(\int_{\frac1p+i\bbR} t^{-z}
m(z ,\xi_2) 
(M_1 \cF_2 \vp)(z ,\xi_2) 
\, dz \bigg)
\, e^{ i\xi_2 x_2}\, d\xi_2 \,,
\end{align*}
and the conclusion follows.
\qed
\ms

\noindent{\bf Remark.}
Notice that we have shown that  if $m$ satisfies the hypotheses of
Thm. \ref{rooney-thm} then
\begin{equation}\label{M-F-mult-eq}
\cC_p^{-1}  T_{m_p}\cC_p (\vp)
= \cF_2^{-1} M_1^{-1} \big( 
m
( M_1 \cF_2\vp) \big) \,.
\end{equation}
Equality \eqref{M-F-mult-eq} clearly explains why we call
the operator $\sT_m$ a Mellin--Fourier multiplier operator.
 It equals a Mellin transformation in the first variable, a Fourier
 transformation in the second variable, followed by multiplication by
 $m$ and then the inverses of the Mellin and Fourier transforms.

Notice that, if $m,\tilde m$ satisfty the assumptions in
Thm. \ref{rooney-thm}, then $\sT_m \sT_{\tilde m}= \sT_{m\tilde m}$.  
%

\bigskip

\section{The Hardy space $\sH^p(D_\beta)$}\label{Har-space-def-section}
In this section we show that the definition of $\sH^p(D_\beta)$ is a very natural one, we
exhibit an explicit formula for the Szeg\H{o} projection $\sP$ and we
prove the sufficient condition in Theorem \ref{t:LpBounds}. 
\ms

For every fixed $z_2$ such that $\big|\log|z_2|^2\big|<\beta-\pt$ let
us consider the half-plane  
$$
\cU_{z_2}:= \{z_1: \re(z_1 e^{-i\log|z_2|^2})>0\}.
$$ 

In order to define $\sH^p(D_\beta)$, it would be natural to consider a
condition that would guarantee
that, for each $z_2$ fixed, the function $f(\cdot,z_2)$ belongs to the
Hardy space on $\cU_{z_2}$. Indeed, this is the case for
the space $\sH^p(D_\beta)$. 

\begin{prop}\label{1st-prop-Hp}
Let $1< p<\infty$ be given.

(i) Let $f\in \sH^p(D_\beta)$. Then, for every fixed $z_2$ such that
$\big|\log|z_2|^2\big|<\beta-\pt$, the function $f_{z_2}:=
f(\cdot,z_2)$ belongs to the (classical) Hardy space on the half-plane
$\cU_{z_2}$, $H^p( \cU_{z_2})$.

(ii) If $f\in \sH^p(D_\beta)$, then $f$ admits boundary values on $\p
D_\beta$, that we 
still denote by $f$, and we have
$$
\| f\|_{\sH^p(D_\beta)} = \| f\|_{L^p(d_b(D_\beta))}\,.
$$
\end{prop}

\begin{proof}
Suppose for the moment that $z_2$ is such that $\log|z_2|^2=-\pt$, so
that $\cU_{z_2}=\{z_1:\im  \,z_1>0\}$. We want to prove that $f_{z_2}$
satisfies 
\begin{equation}\label{half-plane}
\sup_{y>0}\int_\bbR |f_{z_2}(x+iy)|^p dx<\infty.
\end{equation}
From \eqref{GrowthD} it  follows  that $f_{z_2}$ satisfies
\begin{equation}\label{half-plane-2}
 \sup_{y\in(-\pt,\pt)}\int_0^\infty |f_{z_2}(r e^{i(t+\pt)})|^p\ dr<\infty. 
\end{equation}
In \cite{Sed} is proved that conditions \eqref{half-plane} and
\eqref{half-plane-2} actually define the same  space on the upper
 half-plane.  The same argument can be repeated for each $z_2$ such
that $\big|\log|z_2|^2\big|<\beta-\pt$ and conclusion {\it (i)}
follows. 

For the reader's convenience, we postpone the proof of {\it (ii)} to
Section \ref{xxx}.
\end{proof}


\subsection{The Szeg\H o projection $\sP$}

 We now describe the operator $\sP$ and show how it is related to the
 Mellin--Fourier multiplier operators studied in Section \ref{M-F}. 

We first notice that 
\begin{equation*}
d_b(D_\beta)= E^+\cup E^-
\end{equation*}
where
\begin{equation*}
E^+=\left\{(\rho e^{i\beta}, e^{\frac{1}{2}(\beta-\pt)}e^{
    i\theta})\in\bbC^2: \rho\in\bbR, \theta\in [0,2\pi)\right\}
\end{equation*}
and
\begin{equation*}
\ \ \quad E^-=\left\{(\rho e^{-i\beta},
  e^{-\frac{1}{2}(\beta-\pt)}e^{ i\theta})\in\bbC^2: \rho\in\bbR,
  \theta\in [0,2\pi)\right\} 
\end{equation*}
and we remark that both $E^+$ and $E^-$ can be identified with
$\bbR\times\bbT$.  We write 
\begin{align*}
E^+= E^+_0\cup E_1\cup E_2\qquad\text{and }\qquad E^-= E^-_0\cup E_3\cup E_4
\end{align*}
where any component $E_\ell$, $\ell=1,\ldots,4$, can be identified with $\bbR^+\times\bbT$
  and $E^{+}_0, E^-_0$  are sets of $d\sigma$-measure zero. 
Therefore, these latter sets
can be disregarded on what follows. In detail, we have 
\begin{align}\label{El+}
 \begin{split}
&E_1 =\left\{(\rho e^{i\beta}, e^{\frac12(\beta-\pt)}e^{ i\theta}): \rho\in\bbR^+,\theta\in[0,2\pi)\right\};\\
&E_2=\left\{(-\rho e^{i\beta}, e^{\frac12(\beta-\pt)}e^{ i\theta}): \rho \in\bbR^+,\theta\in[0,2\pi)\right\};   \\
&E_0^+=\left\{(0,e^{\frac12(\beta-\pt)}e^{ i\theta}):\theta\in[0,2\pi)\right\}
\end{split}
\end{align}
and
\begin{align}\label{El-}
\begin{split}
&E_3 =\left\{(\rho e^{-i\beta}, e^{-\frac12(\beta-\pt)}e^{ i\theta}): \rho \in\bbR^+,\theta\in[0,2\pi)\right\};\\
&E_4=\left\{(-\rho e^{-i\beta}, e^{-\frac12(\beta-\pt)}e^{ i\theta}): \rho \in\bbR^+,\theta\in[0,2\pi)\right\};\\
&E_0^-=\left\{(0, e^{-\frac12(\beta-\pt)}e^{ i\theta}):\theta\in[0,2\pi)\right\}.
 \end{split}
\end{align}

Consequently, if $\chi_\ell$ denotes the characteristic function of 
$E_\ell$, by linearity, the operator $\sP$ can be decomposed as
\begin{align}\label{P}
\sP
&=\sum_{k,\ell=1,\ldots,4} \chi_k \sP\chi_{\ell}:=
\sum_{,\ell'=1}^4\sP_{k,\ell} \,.
\end{align}
Therefore, we can study the continuity of $\sP$ on $L^p(d_b(D_\beta))$ and $W^{s,2}(d_b(D_\beta))$ by studying the continuity of the
operators $\sP_{k,\ell}$ on the same spaces.

It turns out that the operators $\sP_{k,\ell}$  can be
expressed in terms of the model operators $\sT_m$ 
studied in Section \ref{M-F}.  In order to show this, we set some
notation. 

\begin{defn}\label{mult-expressions}{\rm
  For $z\in\bbC$ and $j\in\bbZ$ we set 
\begin{equation*}
D(z,j)=4\cosh\big[ i\pi(z-\ts{\frac12}) \big] 
\cosh\Big[ i(2\beta-\pi)\big(z-\ts{\frac12}+i(\ts{\frac j2 +\frac14}) \big)\Big]
\end{equation*}
and
\begin{equation*}
S_{\nu_\beta}=\Big\{z\in\bbC: 
\frac{1-\nu_\beta}{2}< z<\frac{1+\nu_\beta}{2}\Big\}.
\end{equation*}

Moreover, for $k,\ell\in\{1,\ldots,4\}$ we set
\begin{equation}\label{(1)}
m_{k,\ell}(z,j)=\frac{e^{i\mu(z-\frac 12)}e^{i\eta(z-\frac12+i\frac j2)}}{D(z,j)} 
\end{equation}
where
\begin{align*}
 &\mu=-\pi\ \text{and } \eta=-(2\beta-\pi)
 &\text{if } (k,\ell)=(1,1);\\
 &\mu=\pi\ \text{and } \eta=-(2\beta-\pi) &\text{if } (k,\ell)=(2,2);\\
 &\mu=\pi\ \text{and } \eta=(2\beta-\pi) &\text{if } (k,\ell)=(3,3);\\
 &\mu=-\pi\ \text{and } \eta=(2\beta-\pi) &\text{if } (k,\ell)=(4,4);\\
 &\mu=0\ \text{and } \eta=-(2\beta-\pi) &\text{if } (k,\ell)\in\{(2,1),(1,2)\};\\
&\mu=0\ \text{and } \eta=(2\beta-\pi) &\text{if } (k,\ell)\in\{(4,3),(3,4)\};\\
 &\mu=-\pi\ \text{and }\ \eta=0 &\text{if } (k,\ell)\in\{(4,1),(1,4)\};\\
 &\mu=\pi\ \text{and }\ \eta=0  &\text{if } (k,\ell)\in\{(3,2),(2,3)\};\\
 &\mu=0\ \text{and }\ \eta=0 &\text{if }  (k,\ell)\in\{(1,3),(3,1),(2,4), (4,2)\}.
\end{align*}
%
%
%
}
\end{defn}

Then, the following theorem holds. 
\begin{thm}\label{Pll'}
For all $k,\ell\in\{1,\ldots,4\}$ it holds that $\sP_{k,\ell}
= \sT_{m_{k,\ell}}$, where the multipliers
$m_{k,\ell}:S_{\nu_\beta}\times \bbZ:\to\bbC $ are as in
Def. \eqref{mult-expressions}.   In particular, for
$\frac{2}{1+\nu_\beta}<p<\frac{2}{1-\nu_\beta}$ the operator $\sP_{k,\ell}$, $k,\ell\in\{1,\dots,4\}$, extend to bounded operator
$$
\sP_{k,\ell} :L^p(d_b(D_\beta))\to L^p(d_b(D_\beta))\,.
$$
\ms
\end{thm}

For the reader's convenience, we prove here only that the multipliers
$m_{k,\ell}$ satisfy the hypothesis of Thm. \ref{rooney-thm}, whereas
we postpone to Section \ref{xxx} the proof that the operators
$\sP_{k,\ell}$ actually are Mellin-Fourier multiplier operators.  

\ms

\begin{proof}
It is immediate to verify that the multiplier $m_{k,\ell}$ satisfies hypothesis $(i)$ in Theorem \ref{rooney-thm}  for every $k,\ell=1,\ldots,4$. Concerning hypothesis $(ii)$, we proceed as follows. Let be $\sigma$ such that $(1-\nu_\beta)/2<\sigma<(1+\nu_\beta)/2$. Then, we want to prove that the function $m_{k,\ell}(\sigma+i\cdot,\cdot): \bbR\times\bbZ\to \bbC$ is a Fourier multiplier on $L^p(\bbR\times\bbT)$. This is obtained by transference (see, e.g., \cite[Chapter 3]{MR2445437} from the fact that the extended function $m_{k,\ell}(\sigma+i\cdot,\cdot):\bbR^2\to\bbC$,
$$
m_{k,\ell}(\sigma+i\xi,\lambda)=\frac{e^{i\mu(\sigma-\frac12+i\xi)}e^{i\eta(\sigma-\frac12+i(\xi+\lambda))}}{4\cosh(i\pi(\sigma-\frac12+i\xi))\cosh(i(2\beta-\pi)(\sigma-\frac12+i(\xi+\lambda+\frac14))}
$$
is a Fourier multiplier on $L^p(\bbR^2)$. This last affirmation is easily proved by showing that the composition of $m_{k,\ell}(\sigma+i\xi,\lambda)$ with the affine change of variables
\begin{equation*}
\begin{cases}
\xi'= \pi\xi \cr
\lambda'=(2\beta-\pi)(\xi+\lambda+1/4)\,,
\end{cases}
\end{equation*} 
is a Mihlin-H\"{o}rmander multiplier (see, e.g, \cite[Chapter 5]{MR2445437}).
This concludes the first part of the proof of Theorem \ref{Pll'}.
\end{proof}

We observe that  Theorem \ref{Pll'} immediately gives proves the positive results in Theorem \ref{t:LpBounds}.  In the next
section we prove the negative results for $\sP$ in both $L^p$ and
Sobolev scales. \ms

\section{$L^p$ and Sobolev irregularity}\label{irregularity}

In this section we provide an explicit counterexample to prove the
negative result of Theorems \ref{t:LpBounds} and \ref{SobolevL2}.

Let $\chi_1$ be the characteristic function of $E_1\subseteq d_b(D_\beta)$ defined in \eqref{El+}. Then, we define on $d_b(D_\beta)$ the function
\begin{equation}\label{function-g}
 g(z_1,z_2):= \chi_1(z_1,z_2)e^{4\beta(\beta+i\log |z_1|)}e^{-(\log |z_1|)^2-\frac12\log |z_1|}. 
\end{equation}
The following lemma is elementary, therefore we omit the proof.
\begin{lem}\label{G-Lp-Sobolev}
The function $g$ belongs to $L^p(d_b(D_\beta))$ and $W^{s,p}(d_b(D_\beta))$ for every $p\in(1,\infty)$ and positive integer $s$.
\end{lem}
Next we focus on the Szeg\H{o} projection of $g$. From \eqref{P} we deduce that $\sP g=\sum_{\ell=1}^4 \sP_{\ell,1}g.$
We want to explicitly compute $\sP_{1,1}g$, using Theorem \ref{Pll'}. 
We have that
\begin{align*}
\cC_2 g(x e^{i\beta},e^{\frac12(\beta-\pt)}e^{ i\theta})&=
e^{4\beta^2} e^{-x^2+i4\beta x } \,,
\end{align*}
so that
\begin{align*}
 \cF \cC_2 g(\xi,j)&= \cF_1\cC_2 g(\xi)\\
 &=e^{4\beta^2}\int_{\bbR} e^{i4\beta x} e^{-x^2} e^{-ix\xi}\ dx	\\
  &=C_\beta e^{-\frac{(\xi-4\beta)^2}{4}} \,,
\end{align*}
where $C_\beta$ is a positive constant depending on $\beta$.
In conclusion,  we obtain that, up to a multiplicative positive constant,
\begin{align}\label{SG}
\begin{split}
 \sP_{1,1}g(x):=\sP_{1,1}g(x e^{i\beta},
 e^{\frac12(\beta-\pt)}e^{i\theta})
&=\cC^{-1}_2\cF^{-1}\bigg[\frac{ e^{-\frac{(\cdot)^2}{4}}}{\cosh(\pi\cdot)\cosh((2\beta-\pi)\cdot)}\bigg] (x).
\end{split}
\end{align}
We recall once again that, by definition, $\sP_{1,1}g$ is supported on
$E_1$; hence $x$ is always assumed to be positive. 

If we define the functions
\begin{align*}
  &f(\xi)= e^{-\frac{\xi^2}{4}};\qquad h(\xi)= \frac{1}{\cosh(\pi\xi)};\qquad r(\xi)= \frac{1}{\cosh((2\beta-\pi)\xi)},
\end{align*}
then,  it is easy to prove that, 
\begin{align*}
  &\cF^{-1}f(x)=\frac{e^{-x^2}}{\sqrt{\pi}};\quad 
\cF^{-1} h(x)=\frac{1}{2\pi}\frac{1}{\cosh(\frac{x}{2})};
\quad \cF^{-1} r(x)=\frac{\nu_\beta}{2\pi}\frac{1}{\cosh(\frac{\nu_\beta x}{2})}
\end{align*}
where $\nu_\beta=\frac{\pi}{2\beta-\pi}$.

Hence, 
\begin{equation}\label{etachimu}
\cF^{-1}
\bigg[\frac{e^{-\frac{(\cdot)^2}{4}}}{\cosh(\pi\cdot)\cosh((2\beta-\pi)\cdot)}\bigg](x)
=C_\beta \int_\bbR\int_\bbR\frac{e^{-s^2}}{\cosh(\frac{\nu_\beta(x-t)}{2})\cosh(\frac{t-s}{2})}\ dsdt.
\end{equation}
We now prove two lemmas.
\begin{lem}\label{decay-T11-infty}
There exist two positive constants $A$ and $B$ such that
\begin{align*}
  A x^{-\frac{\nu_\beta+1}{2}}\leq \sP_{1,1}g(x)\leq B x^{-\frac{\nu_\beta+1}{2}}
\end{align*}
for every $x\in(1,\infty)$.
\end{lem}
\begin{proof}
 From \eqref{SG} and \eqref{etachimu} we get
 \begin{align*}
   x^{\frac{1+\nu_\beta}{2}}\sP_{1,1}g(x)
&=C_\beta x^{\frac{\nu_\beta}{2}}\int_\bbR\int_\bbR \frac{1}{\cosh(\frac{\nu_\beta(\log x-t)}{2})}\frac{e^{-s^2}}{\cosh(\frac{t-s}{2})}\ dsdt\\
   &=C_\beta \int_\bbR\int_\bbR \frac{2}{e^{-\frac{\nu_\beta t}{2}}+e^{-\nu_\beta \log x}e^{\frac{\nu_\beta t}{2}}}\frac{e^{-s^2}}{\cosh(\frac{t-s}{2})} \ ds dt.
 \end{align*}
Now, since $x>1$,
\begin{align*}
 C_\beta \int_\bbR\int_\bbR \frac{2}{e^{-\frac{\nu_\beta t}{2}}+e^{-\nu_\beta \log x}e^{\frac{\nu_\beta t}{2}}}\frac{e^{-s^2}}{\cosh(\frac{t-s}{2})} \ ds dt &\geq C_\beta \int_\bbR\int_\bbR \frac{1}{\cosh(\frac{\nu_\beta t}{2})}\frac{e^{-s^2}}{\cosh(\frac{t-s}{2})}\ ds dt\\
 &=: A.
\end{align*}
Similarly,
\begin{align*}
 C_\beta \int_\bbR\int_\bbR \frac{2}{e^{-\frac{\nu_\beta t}{2}}+e^{-\nu_\beta \log x}e^{\frac{\nu_\beta t}{2}}}\frac{e^{-s^2}}{\cosh(\frac{t-s}{2})} \ ds dt &\leq C_\beta \int_\bbR\int_\bbR \frac{2}{e^{-\frac{\nu_\beta t}{2}}}\frac{e^{-s^2}}{\cosh(\frac{t-s}{2})}\ dsdt\\
 &=: B.
\end{align*}
Thus, the lemma is proved.
\end{proof}
\begin{lem}\label{decay-T11-zero}
 There exist two positive constants $A$ and $B$ such that
 $$
 A x^{\frac{\nu_\beta-1}{2}}\leq \sP_{1,1}g(x)\leq B x^{\frac{\nu_\beta-1}{2}}
 $$
 for every $x\in(0,1)$.
\end{lem}
\begin{proof}
 From \eqref{SG} and \eqref{etachimu} we obtain
 \begin{align*}
  x^{\frac{1-\nu_\beta}{2}}\sP_{1,1}g(x)= C_\beta \int_\bbR\int_\bbR\frac{2}{e^{\nu_\beta\log x}e^{-\frac{\nu_\beta t}{2}}+e^{\frac{\nu_\beta t}{2}}} \frac{e^{-s^2}}{\cosh(\frac{t-s}{2})}\ dsdt.
 \end{align*}
 Now, since $x\in(0,1)$,
\begin{align*}
 C_\beta \int_\bbR\int_\bbR\frac{2}{e^{\nu_\beta\log x}e^{-\frac{\nu_\beta t}{2}}+e^{\frac{\nu_\beta t}{2}}} \frac{e^{-s^2}}{\cosh(\frac{t-s}{2})}\ dsdt&\geq C_\beta \int_\bbR\int_\bbR \frac{1}{\cosh(\frac{\nu_\beta t}{2})}\frac{e^{-s^2}}{\cosh(\frac{t-s}{2})}\ dsdt\\
 &=: A.
\end{align*}
Similarly,
\begin{align*}
 C_\beta \int_\bbR\int_\bbR\frac{2}{e^{\nu_\beta\log x}e^{-\frac{\nu_\beta t}{2}}+e^{\frac{\nu_\beta t}{2}}} \frac{e^{-s^2}}{\cosh(\frac{t-s}{2})}\ dsdt&\leq C_\beta \int_\bbR\int_\bbR \frac{2}{e^{\frac{\nu_\beta t}{2}}}\frac{e^{-s^2}}{\cosh(\frac{t-s}{2})}\ dsdt\\
 &=:B
\end{align*}
and the proof is concluded.
\end{proof}
We point out that the constants $A$ and $B$ in Lemma \ref{decay-T11-infty} and Lemma \ref{decay-T11-zero} are the same. 

The necessary condition in Theorem \ref{t:LpBounds} is immediately deduced.

\begin{proof}[Necessary condition of Theorem \ref{t:LpBounds}]
 It follows easily from Lemmas \ref{decay-T11-infty} and \ref{decay-T11-zero} that $\sP_{1,1}g$ is not in $L^p$ unless $\frac{2}{1+\nu_\beta}<p<\frac{2}{1-\nu_\beta}$. Hence, it follows from Lemma \ref{G-Lp-Sobolev} that the operator $\sP$ cannot be bounded on $L^p$ unless $\frac{2}{1+\nu_\beta}<p<\frac{2}{1-\nu_\beta}$ and the proof is concluded.
\end{proof}

\subsection{The $W^{s,2}$ irregularity}\label{SobolevL2-irregularity}
As mentioned in the introduction, the fractional Sobolev spaces we consider are the ones defined via Fourier transform and Bessel potentials. Namely, the space $W^{s,p}(b_d(D_\beta))$ is defined by the norm
$$
\|f\|^p_{W^{s,p}(d_b(D_\beta))}=\sum_{\ell=1}^4\|f\|^p_{W^{s,p}(E_\ell)}
$$
where 
$$
\|f\|^p_{W^{s,p}(E_\ell)}=\|f\|^p_{W^{s,p}(\bbR\times\bbT)}=\int_{\bbR\times\bbT}\big|\cF^{-1}\big[[1+(\cdot)^2+(\cdot)^2]^{\frac{s}{2}}\cF f(\cdot,\cdot)\big](x,\theta)\big|^p\ dxd\theta.
$$
We point out that in the definition of the Sobolev norm we are integrating with respect to the Lebesgue measure $dxd\theta$ and not with respect to $d\sigma$, the induced measure on $E_\ell$. It is easy to verify that, up to a multiplicative positive constant, these two measure coincides.

Obviously, a different definition of fractional Sobolev spaces can be used, but, due to the nature of the operator $\sP$, this definition is the most natural one for our setting.

We now focus on the Hilbert space $W^{s,2}(d_b(D_\beta))$. In particular, we want to estimate the $W^{s,2}$ of $\sP g$ where $g$ is the function \eqref{function-g} once again. Notice that, since the function $g$ does not depend on the periodic variable on its support, it holds
\begin{equation}\label{Bessel}
\|g\|_{W^{s,2}(\bbR\times\bbT)}^2=\int_{\bbR}\left| \cF^{-1} \big[[1+(\cdot)^2]^{\frac{s}{2}}\cF g(\cdot)\big](x)\right|^2 dx=\|g\|^2_{W^{s,2}(\bbR)},
\end{equation}
An equivalent norm on $W^{s,2}(\bbR)$ is given by
\begin{equation}\label{Gagliardo}
\|g\|^2_{s,2}=\|g\|^2_{W^{[s],2}}+[g]_{s,2}
\end{equation}
where $[s]$ denotes the integer part of $s$, $\|g\|_{W^{[s],2}}$ is the classical Sobolev norm of integer order $[s]$ and $[g]_{s,2}$ is the so-called Gagliardo seminorm, that is,
$$
[g]_{s,2}=\int_\bbR\int_\bbR \frac{|D^{[s]}g(x)-D^{[s]}g(y)|^2}{|x-y|^{1+2s}}\ dxdy.
$$
 
For the equivalence between the norms \eqref{Bessel} and \eqref{Gagliardo} we refer the reader to  \cite{MR781540} where the spaces identified by these two norms are seen as special cases of the more general Triebel spaces. 

We are ready to prove our result. 

\proof[Necessary condition of Theorem \ref{SobolevL2}] Let us assume for the moment that $s\in(0,1)$. Because of \eqref{P} it holds
\begin{align*}
 \int_{d_b(D_\beta)}\int_{d_b(D_\beta)}\frac{|\sP g(x)-\sP g(y)|^2}{|x-y|^{1+2s}}\ dxdy\geq \int_0^\infty\int_0^\infty \frac{|\sP_{1,1}g(x)-\sP_{1,1}g(y)|^2}{|x-y|^{1+2s}}\ dxdy.
\end{align*}
 Suppose now that $x$ and $y$ satisfies $1<y<\alpha x$ where $\alpha$ is a positive number to be fixed. Then, it follows from Lemma \ref{decay-T11-infty} that
 \begin{align*}
  |\sP_{1,1}g(x)-\sP_{1,1}g(y)|&\geq \sP_{1,1}g(y)-\sP_{1,1}g(x)\\
  &\geq A y^{-\frac{1+\nu_\beta}{2}}- B x^{-\frac{1+\nu_\beta}{2}}\\
  &\geq A y^{-\frac{1+\nu_\beta}{2}}- B\big(\frac{ y}{\alpha}\big)^{-\frac{1+\nu_\beta}{2}}\\
  &=y^{-\frac{1+\nu_\beta}{2}}[A-\alpha^{\frac{1+\nu_\beta}{2}} B]\\
  &\geq \frac{A}{2} y^{-\frac{1+\nu_\beta}{2}}
 \end{align*}
if $\alpha$ is chosen such that $\alpha<\big(\frac{A}{2B}\big)^{\frac{2}{1+\nu_\beta}}$. In particular, this implies that $\alpha<\frac{1}{2}$ since $B>A$ and $\frac{1+\nu_\beta}{2}<1$.
 Therefore,
 \begin{align*}
  \int_{0}^\infty\int_{0}^\infty \frac{|\sP_{1,1}g(x)-\sP_{1,1}g(y)|^2}{|x-y|^{1+2s}}\!\!\!&\ dxdy\geq \int_{\{1<y<\alpha x\}}\frac{|\sP_{1,1}g(x)-\sP_{1,1}g(y)|^2}{|x-y|^{1+2s}} \ dxdy\\
  &=\int_0^\infty \int_0^{\arctan\alpha} \frac{|\sP_{1,1}g(\frac{1}{\alpha}+\rho\cos\theta)-\sP_{1,1}g(1+\rho\sin\theta)|^2}{|\frac{1}{\alpha}+\rho\cos\theta-1-\rho\sin\theta|^{1+2s}}\ \rho d\theta d\rho\\
  &\geq \frac{A^2}{4}\int_0^\infty\int_0^{\arctan \alpha}\frac{|(1+\rho\sin\theta)^{-\frac{1+\nu_\beta}{2}}|^2}{|\frac{1}{\alpha}-1+\rho(\cos\theta-\sin\theta)|^{1+2s}}\rho d\theta d\rho.
 \end{align*}
 Notice that the integration in $d\theta$ of this last integral is finite for every $\rho>0$ since $\alpha<\frac{1}{2}$. Thus, we only have to discuss the integration in $d\rho$. It holds,
 \begin{align*}
  \int_{0}^\infty\int_{0}^\infty\frac{|\sP_{1,1}g(x)-\sP_{1,1}g(y)|^2}{|x-y|^{1+2s}}\ dxdy&\geq \frac{A^2}{4}\int_0^\infty\int_0^{\arctan \alpha}\frac{|(1+\rho\sin\theta)^{-\frac{1+\nu_\beta}{2}}|^2}{|\frac{1}{\alpha}-1+\rho(\cos\theta-\sin\theta)|^{1+2s}}\rho d\theta d\rho\\
  &\geq\frac{A^2}{4}\int_0^\infty\int_0^{\arctan \alpha}\frac{|(1+\rho)^{-\frac{1+\nu_\beta}{2}}|^{2}}{|\frac{1}{\alpha}-1+2\rho|^{1+2s}}\rho d\theta d\rho.
 \end{align*}
This integral is finite if and only if $s>-\frac{\nu_\beta}{2}$ and this condition is trivially satisfied since we are assuming $s$ to be positive.

Suppose now that $x$ and $y$ satisfies $1>y>\frac{x}{\alpha}$. Then, from Lemma \ref{decay-T11-zero}, we obtain
\begin{align*}
 |\sP_{1,1}g(x)-\sP_{1,1}g(y)|&\geq \sP_{1,1}g(y)-\sP_{1,1}g(x)\\
 &\geq A y^{\frac{\nu_\beta-1}{2}}-B x^{\frac{\nu_\beta-1}{2}}\\
 &\geq y^{\frac{\nu_\beta-1}{2}}\big[A-\alpha^{\frac{\nu_\beta-1}{2}}B\big]\\
 &\geq \frac{A}{2}y^{\frac{\nu_\beta-1}{2}}
\end{align*}
if $\alpha$ chosen that $\frac{1}{\alpha}<\big(\frac{A}{2B}\big)^{\frac{2}{1-\nu_\beta}}$. Since $B>A$, it follows $\alpha>2$. Therefore,
\begin{align*}
 \int_0^\infty\int_0^\infty \frac{|\sP_{1,1}g(x)-\sP_{1,1}g(y)|^2}{|x-y|^{1+2s}}\ dxdy&\geq \int_{\{\frac{x}{\alpha}<y<1\}}\frac{|\sP_{1,1}g(x)-\sP_{1,1}g(y)|^2}{|x-y|^{1+2s}}\ dxdy\\
 &\geq \int_0^{\frac{1}{2}}\int_{\arctan\frac{1}{\alpha}}^{\pt}\frac{|\sP_{1,1}g(\rho\cos\theta)-\sP_{1,1}g(\rho\sin\theta)|^2}{\rho^{1+2s}|\cos\theta-\sin\theta|} \rho d\theta d\rho\\
 &\geq \frac{A^2}{4}\int_0^{\frac{1}{2}}\int_{\arctan\frac{1}{\alpha}}^{\pt}\frac{\rho^{\nu_\beta-1}(\sin\theta)^{\nu_\beta-1}}{\rho^{sp}|\cos\theta-\sin\theta|^{1+sp}}d\theta d\rho.
\end{align*}
Notice that the integration in $d\theta$ is finite for every $\rho>0$ since $\frac{1}{\alpha}>2$. Instead, the integration in $d\rho$ is finite if and only if $s<\frac{\nu_\beta}{2}$.

Thus, we showed that for $s\in(0,1)$ the Gagliardo seminorm of  $\sP g$ is finite if and only if $0\leq s<\frac{\nu_\beta}{2}$. Hence, because of the equivalence between \eqref{Bessel} and \eqref{Gagliardo}, it holds that $\sP g$ is in $W^{s,2}(\bbR\times\bbT)$ if and only if $0\leq s<\frac{\nu_\beta}{2}$.  This fact, together with Lemma \ref{G-Lp-Sobolev}, proves the necessary condition of Theorem \ref{SobolevL2} in the case $s\in(0,1)$.

Suppose now that $s\geq 1$. If $\sP$ would be bounded on
$W^{s,2}(\bbR\times\bbT)$, then it would be bounded on
$W^{s_1,2}(\bbR\times\bbT)$ as well for every $s_1\in[0,s]$ by
interpolation between $L^2$ and $W^{s,2}$. This is a
contradiction. Hence, we obtain that $\sP$ is an unbounded operator on
$W^{s,2}(\bbR\times\bbT)$ for every $s\geq\frac{\nu_\beta}{2}$ and the
proof is concluded. 
\qed

\subsection{The $W^{s,p}$ irregularity}
By the very same argument we used to prove the $W^{s,2}$ irregularity $\sP$ we prove Theorem \ref{SobolevLp}.

Unlike in the Hilbert setting, we do not have on $W^{s,p}(\bbR)$ an equivalent norm in terms of Gagliardo seminorm similarly to \eqref{Gagliardo}. Nonetheless, there are results which link Bessel potential spaces with the Gagliardo seminorm. 

For $s>0$ and $p\in(1,\infty)$ define the function space 
 \begin{equation*}
 \cW^{s,p}(\bbR)=\left\{f\in \cW^{[s],p}(\bbR) : [f]_{s,p}:=\int_{\bbR}\int_{\bbR}\frac{|D^{[s]}f(x)-D^{[s]}f(y)|^p}{|x-y|^{1+sp}}<\infty\right\},
 \end{equation*}
 where $[s]$ denotes the integer part of $s$, $D^{[s]}f$ is the $[s]$-derivative of the function $f$ and $\cW^{[s],p}(\bbR)$ is the classical Sobolev space of integer order $[s]$. The space $\cW^{s,p}(\bbR)$ is endowed with the norm
 $$
 \| f\|_{\cW^{s,p}(\bbR)}^p=\|f\|_{\cW^{[s].p}(\bbR)}^p+[f]_{s,p}.
 $$
 
As in the Hilbert setting, the spaces $W^{s,p}(\bbR)$  and $\cW^{s,p}(\bbR)$ can be seen as special cases of Triebel spaces \cite{MR781540}.
 
As we mentioned, unless $p=2$, the spaces $W^{s,p}(\bbR)$ and $\cW^{s,p}(\bbR)$ do not coincide; nonetheless, the following proposition holds.
 
 \begin{prop}\label{InclusionTriebel}
 Let be $s>0$ and $p\in(1,\infty)$. Then, for every $\eps>0$, 
 \begin{equation}\label{inclusion}
 W^{s+\eps,p}(\bbR)\subseteq \cW^{s,p}(\bbR)\subseteq
 W^{s-\eps,p}(\bbR) \,,
 \end{equation}
 where $A\subseteq B$ denotes continuous inclusion. 
 
 Moreover, if $p>2$, it holds
 \begin{equation}\label{inclusion2}
 W^{s,p}(\bbR)\subseteq \cW^{s,p}(\bbR).
 \end{equation}
 \end{prop}
 \begin{proof}
 See \cite{MR781540}.
 \end{proof}

 Exploiting this last proposition we prove Theorem \ref{SobolevLp}.
\proof[Proof of Theorem \ref{SobolevLp}]
   Computing the Gagliardo seminorm of $\sP g$ for $s\in(0,1)$ as in the
 Hilbert setting, we obtain that $\|\sP g\|_{\cW^{s,p}(\bbR)}<\infty$
 if and only if the stronger condition 
 $$
 -\frac{\nu_\beta}{2}<s+\frac12-\frac1p<\frac{\nu_\beta}{2}
 $$
 holds.
 Then, using \eqref{inclusion} we conclude the proof in the case
 $s\in(0,1)$. If $s\geq1$ we conclude using interpolation as in the
 Hilbert setting. 
 In we assume also $p\geq 2$ we repeat the same argument using the
 stronger inclusion \eqref{inclusion2}.  \ms
\qed

\section{Sobolev regularity}\label{Sobolev-regularity}
In this section we prove the positive result in Theorem
\ref{SobolevL2}. As in the $L^p$ setting, we deduce the boundedness of
$\sP$ from the boundedness of the operators $\sP_{k,\ell}$,
$k,\ell=1,\ldots,4$. 

To prove our result, we obtain an integral representation for the
general operator $\sP_{k,\ell}$ and then we exploit some classical
results for Calder\'{o}n--Zygmund singular integral operators \cite{MR1800316, MR2445437,MR2463316}.

From Theorem \ref{Pll'}
 we get
\begin{equation*}
 \sP_{k,\ell}
=\chi_k (\cC_2^{-1}T_{m_\mu}\cC_2)\circ(\cC_2^{-1}T_{m_\eta}\cC_2)\chi_\ell
\end{equation*}
where $T_{m_\mu}$ and $T_{m_\eta}$ are Fourier multiplier operators on
$L^2(\bbR\times\bbT)$ associated to the multipliers  
\begin{align*}
 m_\mu(\xi)=\frac{e^{\mu\xi}}{2\cosh(\pi\xi)} 
&\qquad \text{ and } \qquad m_\eta(\xi,j)
=\frac{e^{\eta(\xi-\frac j2)}}{2\cosh((2\beta-\pi)(\xi-\frac j2-\frac14)} 
\end{align*}
with $|\mu|=\pi$ or $\mu=0$ and $|\eta|=2\beta-\pi$ or $\eta=0$
according to the scheme in Definition \ref{mult-expressions}.   

Thus,
we deduce the boundedness $\sP_{k,\ell}: W^{s,2}(d_b(D_\beta))\to
W^{s,2}(d_b(D_\beta))$, from the boundedness of the operators
\begin{equation}\label{model-Pa}
\begin{split}
P_a & =\chi_+ (\cC_2^{-1}T_{m_a}\cC_2)\chi_+\\
Q_a & =\chi_+ (\cC_2^{-1}T_{M_a}\cC_2)\chi_+\\
\end{split}
\end{equation}
where $\chi_+$ denotes 
the characteristic function
of $(0,\infty)$ and, for $ |a|\geq\pi$,
\begin{equation*}
\begin{split}
m_a(\xi)& =\frac{e^{a\xi}}{2\cosh(a\xi)},\\
M_a(\xi,j) & = e^{\frac a4} m_a(\xi {\textstyle{-\frac j2 -\frac14}}
)\,.
\end{split}
\end{equation*}

\begin{remark}{\rm
It should be noticed that the operators $P_a$ and $Q_a$ involve the
multiplication by the characteristic function $\chi_+$, that obviously
destroys some smoothness.  However, we only need to consider the case
of Sobolev spaces $W^{s,2}$, with $0<s<\frac{\nu_\beta}{2}<\frac12$
and the multiplication by $\chi_+$ preserves
the Sobolev spaces for such small values of $s$, as we will see.

We also point out that the multipliers $m_a$ and $M_a$ are not models for $m_\mu$ and $m_\eta$ respectively when $\mu=\eta=0$. Nonetheless, these
cases are the easiest to deal  with and the techniques we use can be
easily adapted to these situations. 
}
\end{remark}

\begin{thm}\label{Sobolev-Pa}
 The operator $P_a$ defines a bounded operator 
$$
P_a:
 W^{s,2}(\bbR\times\bbT)\to W^{s,2}(\bbR\times\bbT)
$$
 for every  $s\in\big[0,\frac{\pi}{2|a|}\big)$. 
\end{thm}

In order to prove the theorem, we need a few preliminary results.
\begin{lem}\label{Fourier-transf-pv}
 Let be $|a|\geq\pi$ and $\kappa\in[0,\frac{\pi}{2|a|})$ and consider the
 tempered distribution $U_{a,\kappa}$ defined by, for $g$ in  the Schwartz space $\cS(\bbR)$, 
 $$
 \left<U_{a,\kappa},g\right>
=\lim_{\eps\to0}
\int_{\eps<|\frac{\pi\xi}{2a}|}\frac{e^{-\kappa t}}{\sinh(\frac{\pi
     t}{2a})} g(t)\, dt. 
 $$
 Then, the Fourier transform $\widehat{U_{a,\kappa}}$ is given by 
\begin{equation}\label{F-T-Uka}
 \widehat{U_{a,\kappa }}(\xi)=2|a| i\tanh\big(a(\xi-i\kappa )\big). 
\end{equation} 
 \end{lem}
 \begin{proof}
Writing
$$
\int_{\eps<|\frac{\pi t}{2a}|}\frac{e^{-\kappa t}}{\sinh(\frac{\pi
     t}{2a})} g(t)\, dt
= \int_{\eps<|\frac{\pi t}{2a}|\le 1}\frac{e^{-\kappa t}g(t)- g(0)}{\sinh(\frac{\pi
     t}{2a})} \, dt + \int_{1<|\frac{\pi t}{2a}|}\frac{e^{-\kappa t}}{\sinh(\frac{\pi
     t}{2a})} g(t)\, dt
$$
it is easy to check that $U_{a,\kappa }$ is in fact a well-defined tempered
distribution.

Next, a standard contour integration
argument now shows that, for $\zeta\in\bbC$ with $|\Im\zeta|<1$, 
$$
\int_{\eps<|t|<R} \frac{e^{-i\zeta t}}{\sinh t}\, dt = i\pi
\tanh({\textstyle{\frac{\pi}{2}}}\zeta) + o(1) 
\, ,
$$
as $\eps\to0^+$ and $R\to+\infty$, so that
\begin{equation}\label{R-eps-contour}
\int_{\eps<|\frac{\pi\xi}{2a}|<R} \frac{e^{-i\zeta t}}{\sinh (\frac{\pi
     t}{2a}) }\, dt = 2 |a| i
\tanh(a\zeta) + o(1) \, ,
\end{equation}
as $\eps\to0^+$ and $R\to+\infty$. 
Therefore, setting $\zeta=\xi-i\kappa $ we obtain \eqref{F-T-Uka} with
$0\le \kappa <\frac{\pi}{2|a|}$. 
 \end{proof}

\begin{cor}\label{corollary}
 The following facts hold.
 \begin{itemize}
  \item[$(i)$] The convolution operator $T_{a,\kappa }: g\mapsto U_{a,\kappa }\ast g$,
    densely defined on $\cS(\bbR)$, extends to a bounded operator $L^2\to L^2$ for every
    $\kappa \in[0,\frac{\pi}{2|a|})$. 
  \smallskip
  \item[$(ii)$] The Fourier multiplier operator $T_{m_a}$ is given by
$$
 T_{m_a}
={\textstyle{\frac12}} \big( I +  {\textstyle{\frac{1}{2|a|i}}}  
T_{a,0}
\big) \,. 
 $$
 \item[$(iii)$] The operator $T_{a,0}$ is a Calder\'{o}n--Zygmund
   singular integral operator. Therefore, for every
   $0<\varepsilon<R<\infty$, the truncated operator 
 $$
 T_{a,0,(\varepsilon,R)}g(x):=\int_{\eps<|\frac{\pi t}{2a}|<R} \frac{g(x-t)}{\sinh(\frac{\pi t}{2a})} \, dt
 $$
 densely defined on $\cS(\bbR)$ extends to a bounded operator
 $L^2\to L^2$ with operator norm independent of $\varepsilon$ and
 $R$. Moreover, 
 $$
\lim_{\eps\to 0^+ \atop R\to+\infty} T_{a,0,(\varepsilon,R)}g= T_{a,0}g
 $$
 in $L^2(\bbR)$.
  \end{itemize}
\end{cor}
\begin{proof} 
  $(i)$ is obvious, since $\widehat{U_{a,\kappa }}\in L^\infty(\bbR)$ for
  $\kappa \in[0,\frac{\pi}{2|a|})$.
Since 
  $$
  \widehat{U_{a,0}}(\xi)=  2|a| i 
  \tanh(a\xi)=  2|a| i 
\Big(\frac{e^{a\xi}}{\cosh(a\xi)}-1\Big)
=  4|a| i 
\big( m_a(\xi) - {\textstyle{\frac12}} \big)\,,
  $$
$(ii)$ also follows. Part $(iii)$ follows from the standard theory of
Calder\'{o}n--Zygmund singular integral operators and we refer the
reader, for instance, to \cite[Chapter 5]{MR1800316}. 
\ms
\end{proof}

From the previous corollary, we finally obtain an integral
representation for the operator $P_a$.  
For $|a|\ge\pi$, we now set 
\begin{equation}\label{Lambda-eps-R}
\Lambda_{a,(\eps,R)}:= \chi_+\cC^{-1}_2 T_{a,0,(\varepsilon,R)}\cC_2\chi_+ \,.
\end{equation}

\begin{cor}\label{corollary2}
The following facts hold.
\begin{itemize}
\item[$(i)$] there exists a constant $C>0$, independent of
$\eps,R>0$ such that 
$$
\| \Lambda_{a,(\eps,R)} f \|_{L^2(\bbR\times\bbT)} \le C \| f
\|_{L^2(\bbR\times\bbT)}\,;
$$
\item[$(ii)$]$\lim_{\eps\to 0^+ \atop R\to+\infty} \Lambda_{a,(\eps,R)} f 
:= \Lambda_a f$
exists in $L^2(\bbR\times\bbT)$-norm, for every test function $f$,
and it defines a bounded operator
$
\Lambda_a :L^2(\bbR\times\bbT)\to
L^2(\bbR\times\bbT)$\,;
\smallskip
\item[$(iii)$] For every $f\in L^2(\bbR\times\bbT)$ the operator $P_a$ is given by
\begin{align*}
P_a f  & = {\textstyle{\frac12}}  (\chi_+f)
+  {\textstyle{\frac{1}{2|a|i}}}  
\chi_+(\cC_2^{-1} T_{a,0} \cC_2)(\chi_+f)
  \\
& =:  {\textstyle{\frac12}} \Lambda_+ f+ 
{\textstyle{\frac{1}{2|a|i}}}  \Lambda_a f\,;
\end{align*}
\smallskip
\item[$(iv)$] for $(\xi,j)\in\bbR\times\bbZ$ and $f\in
  L^2(\bbR\times\bbT) $,
  $$
  \cF(\Lambda_a f)(\xi,j)= \begin{cases} 
      \Lambda_a(\cF(\chi_+ f))(\xi,j) & \text{if } \xi> 0 \\
      \phantom{\int}\\
      \Lambda_a(\widetilde{\cF(\chi_+ f))}(-\xi,j) & \text{if } \xi< 0\,,
   \end{cases}
  $$
where $\widetilde{\cF(\chi_+ f))}(x,\theta):=\cF(\chi_+ f))(-x,\theta)$.
\end{itemize}
\end{cor}
\proof 
From $(iii)$ of Corollary \ref{corollary} and recalling that $\cC_2:
L^2((0,\infty),\times\bbT)\to L^2(\bbR\times\bbT)$ is an isometry, we
get 
\begin{align*}
\int_\bbT\int_\bbR \big|\Lambda_{a,(\eps,R)} f (x,\theta)\big|^2 \,
dxd\theta
&= \int_\bbT\int_0^{+\infty} 
\big| \cC^{-1}_2 T_{a,0,(\varepsilon,R)}\cC_2 (\chi_+f)(x,\theta)\big|^2 \, dxd\theta \\
&=\int_\bbT\int_\bbR
\big|  T_{a,0,(\varepsilon,R)}\cC_2 (\chi_+f)(x,\theta)\big|^2 \, dxd\theta\\
&\leq C\int_\bbT\int_\bbR \big|\cC_2 (\chi_+f)(x,\theta) \big|^2\ dxd\theta \\
&= C\int_\bbT\int_0^{+\infty}|\chi_+(x) f(x,\theta)|^2\ dxd\theta \\
&\leq C \int_\bbT\int_\bbR \big|f(x,\theta)\big|^2\ dxd\theta
\end{align*}
as we wished. Thus, $(i)$ is proved. With a similar argument, $(ii)$
is also proved. For $(iii)$ the proof is easily deduced from these
results, $(ii)$ of Corollary \ref{corollary} and \eqref{model-Pa}.  

Finally, observing that the integrals below converge absolutely and
interpreting the limits as limits in the $L^2$-norm, 
\begin{align*} 
\cF (\Lambda_a f)(\xi,j) 
&= \lim_{\eps\to 0^+ \atop R\to +\infty} \cF (\Lambda_{a,(\eps,R)}
f)(\xi,j) \\
 &= \lim_{\eps\to 0^+ \atop R\to +\infty}
\int_{\eps<|\frac{\pi t}{2a}|<R}\frac{e^{-\frac t2}}{\sinh(\frac{\pi
    t}{2a})} \int_\bbR \chi_+(x) \cF_2 f(xe^{-t},j) e^{-ix\xi}\, dx
dt\\ 
 &= \lim_{\eps\to 0^+ \atop R\to +\infty} 
\int_{\eps<|\frac{\pi t}{2a}|<R}\frac{e^{-\frac t2}}{\sinh(\frac{\pi
    t}{2a})} \cF\Big( (\chi_+ f)\big((\cdot) e^{-t},\cdot\big) \Big)(\xi,j)\, dt\\
  &= \lim_{\eps\to 0^+ \atop R\to +\infty} 
\int_{\eps<|\frac{\pi t}{2a}|<R}\frac{e^{\frac t2}}{\sinh(\frac{\pi
    t}{2a})}\cF (\chi_+ f)(\xi e^t,j)\, dt \,.
\end{align*}
 Suppose now that $\xi>0$. Then, 
\begin{align*}
 \int_{\eps<|\frac{\pi t}{2a}|<R}\frac{e^{\frac t2}}{\sinh(\frac{\pi
    t}{2a})}\cF (\chi_+ f)(\xi e^t,j)\, dt
&=e^{-\frac{\log\xi}{2}}\int_{\eps<|\frac{\pi t}{2a}|<R}\frac{e^{\frac{\log\xi+t}{2}}}{\sinh(\frac{\pi
    t}{2a})}\cF (\chi_+ f)(e^{	\log\xi+t},j)\, dt \\
    &=\Lambda_{a,(\varepsilon,R)}(\cF(\chi_+ f))(\xi,j).
\end{align*}
Therefore, if $\xi>0$,
$$
\cF(\Lambda_a f)(\xi,j)=\Lambda_{a}(\cF(\chi_+ f))(\xi,j)\,,
$$
as we wished to show. Similarly, if $\xi<0$, we get
\begin{align*}
  \int_{\eps<|\frac{\pi t}{2a}|<R}\frac{e^{\frac t2}}{\sinh(\frac{\pi
    t}{2a})}\cF (\chi_+ f)(\xi e^t,j)\, dt
&=e^{-\frac{\log(-\xi)}{2}}\int_{\eps<|\frac{\pi t}{2a}|<R}\frac{e^{\frac{\log(-\xi)+t}{2}}}{\sinh(\frac{\pi
    t}{2a})}\cF (\chi_+ f)(-e^{	\log(-\xi)+t},j)\, dt \\
    &=\Lambda_{a,(\varepsilon,R)}\widetilde{(\cF(\chi_+ f))}(\xi,j)
\end{align*}
and $(iv)$ is finally proved.
\qed
\ms

We now have all the ingredients to prove Theorem \ref{Sobolev-Pa}. The
boundedness of $P_a$ will follow from the following lemma. 

\begin{lem}\label{lem1-sobolev}
 Let be $|a|\geq\pi$. Then, for every $s\in[0,\frac{\pi}{2|a|})$,
 \begin{itemize}
  \item[$(i)$] the operator $\Lambda_+$ defines a bounded
    operator 
$\Lambda_+ : W^{s,2}(\bbR\times\bbT)\to W^{s,2}(\bbR\times\bbT)$\,; 
  \smallskip
    \item[$(ii)$] the operator $\Lambda_a$ defines a bounded
    operator 
$\Lambda_a: W^{s,2}(\bbR\times\bbT)\to W^{s,2}(\bbR\times\bbT)$\,.
 \end{itemize}
\end{lem}

Before proving the lemma, we fix some notation and recall the
definition of the classical Hilbert transform. We refer the reader,
for instance, to \cite[Chapter 4]{MR2445437}.  

Let be $f:\bbR\times\bbT\to \bbC$ a test function. The Hilbert
transform 
with respect to the
first variable   is defined by the formula 
\begin{equation*}
 \cF_1(\cH_{1}f)(\xi,\theta)=-i\sgn(\xi)\cF_1(f)(\xi,\theta) \,,
\end{equation*}
where  as usual, $\cF_1$ denotes the Fourier transform with respect to  the
first variable.

If $p\in(0,\infty)$ and $w(x)$ is a non-negative real function, then
we denote by $L^p(w)$ the space of functions such that 
$$
\int_\bbR |g(x)|^p w(x)\ dx<\infty.
$$
It is a well-known result that the Hilbert transform extends to
a bounded operator $L^p(w)\to L^p(w)$ whenever $w$ is a weight
in the
Muckenhoupt class $A_p$  (see, e.g., \cite{MR2463316}). 

We are now ready to prove the lemma.
\proof[Proof of Lemma \ref{lem1-sobolev}] In order to prove the
boundedness of $\Lambda_+$, by Plancherel's formula, we need to
estimate 
$$
\|\Lambda_+ f\|_{W^{s,2}(\bbR\times\bbT)}^2
=\frac{1}{2\pi}\sum_{j\in\bbZ}\int_{\bbR}(1+\xi^2+j^2)^s|\cF(\chi_+ f)(\xi,j)|^2\ d\xi. 
$$
Notice that 
\begin{align*}
 \cF(\chi_+ f)(\xi,j)
&=\frac12\int_{-\infty}^\infty (1+\sgn(x))\cF_2 f(x,j)e^{-ix\xi}\ d\xi\\
 &=\frac{1}{2}\Big(\cF f(\xi,j)- i\cH_1\cF f(\xi,j) \Big).
\end{align*}
Moreover, since $s<\frac{\pi}{2|a|}<\frac12$, it holds
$$
(1+\xi^2+j^2)^s\leq1+\xi^{2s}+j^{2s}
$$
and the weight $w(\xi)=\xi^{2s}$ turns out to belong to the $A_2$
class (see \cite[Example 9.1.8]{MR2463316}). 

Therefore,
\begin{align*}
 \|\chi_+ f\|^2_{W^{s,2}}
&\leq \frac {1}{2\pi}\sum_{j\in\bbZ}\int_\bbR(1+\xi^2+j^2)^s|\cF(\chi_+ f)(\xi,j)|^2\ d\xi\\
 &\leq C \sum_{j\in\bbZ}\int_\bbR(1+\xi^2+j^2)^s\big|\cF f(\xi,j)-i\cH_1\cF f(\xi,j) \big|^2\ d\xi\\
 &\leq C\|f\|^2_{W^{s,2}} 
+C\sum_{j\in\bbZ}\int_\bbR (1+\xi^{2s}+j^{2s})|\cH_1 \cF f(\xi,j)|^2\ d\xi
\end{align*}
From the boundedness of $\cH_1$ in $L^2$, in $L^2(\xi^{2s})$ and
exploiting also Plancherel's formula, we obtain 
\begin{align*}
 \sum_{j\in\bbZ}\int_\bbR (1+\xi^{2s}+j^{2s})|\cH_{1}\cF f(\xi,j)|^2\
 d\xi
&\leq C \sum_{j\in\bbZ}\int_\bbR (1+\xi^{2s}+j^{2s})|\cF f(\xi,j)|^2\ d\xi\\
 &\leq C \sum_{j\in\bbZ}\int_\bbR (1+\xi^{2}+j^{2})^s|\cF f(\xi,j)|^2\ d\xi\\
 &\leq C\|f\|^2_{W^{s,2}}\, ,
 \end{align*}
as we wished to show.

Analogously, for the operator $\Lambda_a$ we have
\begin{align} \nonumber \label{Ma}
 \sum_{j\in\bbZ}\int_\bbR (1+\xi^2+j^2)^s |\cF (\Lambda_a
 f)(\xi,j)|^2\
 d\xi&\leq\sum_{j\in\bbZ}\int_{\bbR}(1+\xi^{2s}+j^{2s})|\cF (\Lambda_a
 f)(\xi,j)|^2\ d\xi\\  
 &\leq\sum_{j\in\bbZ}\left[\int_0^\infty+\int_{-\infty}^0
 \right](1+\xi^{2s}+j^{2s}) |\cF (\Lambda_a f)(\xi,j)|^2\ d\xi. 
\end{align}
We now focus on the integration over the positive $\xi$'s. From
Corollary \ref{corollary2}  (iv), for $\xi>0$, using the fact that
$s<\frac{\pi}{2|a|}$ we have
\begin{align*}
\xi^{s}\cF (\Lambda_a f)(\xi,j) 
& = \xi^{s} \Lambda_a( \cF(\chi_+f)) (\xi,j)  \\
& = \xi^{s}  e^{-\frac{\log \xi}{2}} \lim_{\eps\to0^+}
\int_{\eps<|\frac{\pi t}{2a}|} \frac{1}{\sinh \big( \frac{\pi
    t}{2a}\big)} \cF(\chi_+f) (e^{\log \xi -t},j)
e^{\frac{\log\xi-t}{2}}\, dt
\\
& =  e^{-\frac{\log \xi}{2}} \lim_{\eps\to0^+}
\int_{\eps<|\frac{\pi t}{2a}|} \frac{e^{st}}{\sinh \big( \frac{\pi
    t}{2a}\big)}  e^{s(\log\xi-t)} \cF(\chi_+f) (e^{\log \xi -t},j)
e^{\frac{\log\xi-t}{2}}\, dt
\\
& =  e^{-\frac{\log \xi}{2}} \lim_{\eps\to0^+}
\int_{\eps<|\frac{\pi t}{2a}|} \frac{e^{st}}{\sinh \big( \frac{\pi
    t}{2a}\big)}  e^{s(\log\xi-t)} \cF(\chi_+f) (e^{\log \xi -t},j)
e^{\frac{\log\xi-t}{2}}\, dt
\\
& = \cC^{-1}_2T_{a,s}\cC_2 \big((\cdot)^s \cF(\chi_+ f)\big)(\xi,j) \,,
\end{align*}
where $((\cdot)^s \cF(\chi_+ f))(\xi,j)=\xi^s \cF(\chi_+ f)(\xi,j)$. 
Therefore, from the $L^2$ boundedness of $T_{a,s}$ and $\Lambda_+$, we get
\begin{align*}
  \sum_{j\in\bbZ}\int_0^\infty \xi^{2s}|\cF (\Lambda_a f)(\xi,j)|^2\
  d\xi
&\leq C\sum_{j\in\bbZ}\int_\bbR |\cC_2\big((\cdot)^s \cF(\chi_+ f)\big)(\xi,j)|^2\ d\xi\\
 &\leq C\sum_{j\in\bbZ}\int_0^\infty \xi^{2s}|\cF(\chi_+ f)(\xi,j)|^2\ d\xi\\
  &\leq C\sum_{j\in\bbZ}\int_\bbR (1+\xi^2+j^2)^s\cF(\chi_+ f)(\xi,j)|^2\ d\xi\\
 &\leq C\|f\|^2_{W^{s,2}} \,.
\end{align*}

%
%
Similarly, from Corollary \ref{corollary2} and the $L^2$ boundedness of $\Lambda_a$ and $\Lambda_+$, we deduce
\begin{align*}
 \sum_{j\in\bbZ}(1+j^{2s})\int_0^\infty |\cF (\Lambda_a f)(\xi,j)|^2\ d\xi&=\sum_{j\in\bbZ}(1+j^{2s}) \int_0^\infty | \Lambda_a (\cF(\chi_+f))(\xi,j)|^2\ d\xi\\
 &\leq C\sum_{z\in\bbZ}(1+j^{2s})\int_\bbR |\cF(\chi_+ f)(\xi,j)|^2\ d\xi\\
 &\leq C\sum_{j\in\bbZ}\int_\bbR (1+\xi^2+j^2)^s\cF(\chi_+ f)(\xi,j)|^2\ d\xi\\
 &\leq C\|f\|^2_{W^{s,2}}.
\end{align*}

Hence, we finally estimated the integration over the positive $\xi$'s
in \eqref{Ma}. A completly analogous argument allows us to estimate the over the
negative $\xi$'s. Therefore, the proof is concluded. 
\qed
\proof[Proof of Theorem \ref{Sobolev-Pa}]
It follows from Corollary \ref{corollary}2 and Lemma \ref{lem1-sobolev}.\ms
\qed

In order to conclude the proof of Theorem \ref{SobolevL2}, in analogy with
Theorem \ref{Sobolev-Pa},  we prove the following.

\begin{thm}\label{Sobolev-Qa}
 The operator $Q_a$ defines a bounded operator 
$$
Q_a:
 W^{s,2}(\bbR\times\bbT)\to W^{s,2}(\bbR\times\bbT)
$$
 for every  $s\in\big[0,\frac{\pi}{2|a|}\big)$. 
\end{thm}
\begin{proof}
By density, we consider a test function $f$ on $\bbR\times\bbT$ which is a trigonometric polynomial in the second variable, that is,
$$
f(x,\theta)=\sum_{j=-N}^N f(x,j)e^{ij\theta}.
$$
Then, setting
$$
g_j(x,\theta):= x^{-i(\textstyle{\frac j2+\frac 14})}(\chi_+ f)(x,\theta).
$$
and using $(ii)$ of Corollary \ref{corollary}, we see that
\begin{align*}
Q_af (x,\theta)
&=\frac{e^{\frac a4}}{(2\pi)^2} \chi_+(x)  e^{-\frac{\log x}{2}} \sum_{j=-N}^N  e^{ij\theta}
\int_{\bbR}
\frac{e^{a(\xi-\frac j2 -\frac14)}}{2\cosh(a(\xi-\textstyle{\frac j2-\frac14})}
\cF(\cC_2(\chi_+ f))(\xi,j) e^{i(\log x)\xi}\, d\xi \\
&= \frac{e^{\frac a4}}{(2\pi)^2}\chi_+(x) e^{-\frac{\log x}{2}} \sum_{j=-N}^N
e^{ij\theta}\int_{\bbR} \frac{e^{a\xi}}{2\cosh(a\xi)}
\cF(\cC_2(\chi_+f)) (\xi+\textstyle{\frac j2+\frac14},j) 
e^{i(\log x)(\xi+\frac j2+\frac14)}\, d\xi \\
&=\frac{e^{\frac a4}}{(2\pi)^2}
\chi_+(x) e^{-\log x(\frac12-\frac i4)}\sum_{j=-N}^N
e^{ij(\theta+\frac{\log x}{2})}\int_\bbR \frac{e^{a\xi}}{2\cosh(a\xi)}
\cF(\cC_2 g_j)(\xi,j) e^{i(\log x)\xi}\, d\xi \\
&=\frac{e^{\frac a4}}{2\pi}
\chi_+(x) e^{-\log x(\frac12-\frac i4)}\sum_{j=-N}^N
e^{ij(\theta+\frac{\log x}{2})} T_{m_a} \big[
\cF_2(\cC_2 g_j)(\cdot,j)\big] (\log x) \\
& = \frac{e^{\frac a4}}{2} \Lambda_+f(x,\theta)
+ \frac{e^{\frac a4}}{8\pi|a|i}
\chi_+(x)e^{-\log x(\frac12-\frac
  i4)}\sum_{j=-N}^N e^{ij(\theta+\frac{\log x}{2})} T_{a,0}\big(\cF_2
\cC_2 g_j\big)(\log x,j) \\ 
 &=: C\Lambda_+ f(x,\theta)+ C'\widetilde{\Lambda_a}f
 (x,\theta) \,,
\end{align*}
where $C,C'$ are positive constants.

Therefore, to conclude the proof we only have to prove boundedness of
the operator $\widetilde{\Lambda_a}$. Arguing as in the proof of
Corollary \ref{corollary2} (iv), for $\xi>0$ we get  
\begin{align*}
\cF\widetilde{\Lambda_a} f(\xi,j)
&=\int_\bbR \chi_+(x) \lim_{\eps\to 0} \int_{|\frac{\pi t}{2a}|>\eps}
\frac{e^{it(\frac j2+\frac14)}e^{-\frac t2}}{\sinh(\frac{\pi
    t}{2a})}\cF_2 (\chi_+ f)(xe^{-t},j)\, dt\, e^{-ix\xi}\, dx\\
&= e^{-i(\frac j2+\frac14)\log\xi} \cC_2^{-1} T_{a,0} \cC_2 \big((\cdot)^{i(\frac j2+\frac14)}\cF(\chi_+ f)\big)
(\xi,j)\\
&= e^{-i(\frac j2+\frac14)\log\xi}\Lambda_a\big( (\cdot)^{i(\frac j2+\frac14)}\cF(\chi_+ f)\big)(\xi,j).
\end{align*}

Therefore, similarly to \eqref{Ma}, we obtain
\begin{align*}
 \sum_{j\in\bbZ}\int_0^\infty (1+\xi^2+j^2)^s 
|\cF \widetilde{\Lambda_a}f(\xi,j)|^2\ d\xi
&=\sum_{j\in\bbZ}\int_0^\infty(1+\xi^2+j^2)|^s\Lambda_a\big( (\cdot)^{i(\frac j2+\frac14)}\cF(\chi_+ f)\big)(\xi,j)|^2\, d\xi\\
&\leq \sum_{j\in\bbZ}\int_0^\infty(1+\xi^2+j^2)|^s\big( (\cdot)^{i(\frac j2+\frac14)}\cF(\chi_+ f)\big)(\xi,j)|^2\, d\xi\\
&=\sum_{j\in\bbZ}\int_0^\infty(1+\xi^2+j^2)|^s\cF(\chi_+ f)(\xi,j)|^2\, d\xi\\
&\leq \| f\|^2_{W^{s,2}(\bbR\times\bbT)}.
\end{align*}
With a similar argument we estimate the integration over negative
$\xi$'s and we are done.
\end{proof}
Finally, Theorems \ref{Sobolev-Pa} and \ref{Sobolev-Qa} imply the boundedness $\sP_{k,\ell}: W^{s,2}(d_b(D_\beta))\to W^{s,2}(d_b(D_\beta))$, hence, the boundedness of the Szeg\H{o} projection $\sP$. The proof of Theorem \ref{SobolevL2} is complete.

\section{Proof of Theorem \ref{Pll'}}\label{xxx}
In order to conclude the proof of Theorem \ref{Pll'}, we exploit the
space $\sH^2(D'_\beta)$ studied in \cite{M}. We recall here the main
facts needed for our purposes and we refer the reader to \cite{M} for
the proofs. 
\subsection{The Hardy space $\sH^2(D'_\beta)$}
Consider the domain
\begin{equation*}
D'_\beta=\left\{(z_1,z_2)\in\bbC^2: \big|\im\,z_1-\log|z_2|^2\big|<\pt,\big|\log|z_2|^2\big|<\beta-\pt\right\}
\end{equation*}
and its distinguished boundary 
\begin{equation*}
d_b(D'_\beta)=\left\{(z_1,z_2)\in\bbC^2:
  \big|\im\,z_1-\log|z_2|^2\big|=\pt,\big|\log|z_2|^2\big|=\beta-\pt\right\}. 
\end{equation*}
Then, the space $\sH^2(D'_\beta)$ is defined as the function space
$$
  \sH^2(D'_\beta)=\Big\{f\in\Hol(D'_\beta):
  \|f\|^2_{\sH^2(D'_\beta)}
=\sup_{(t,s)\in [0,\pt)\times[0,\beta-\pt)}\|f\|^2_{L^2(d_b(D'_\beta))}<\infty\Big\},
$$
where, 
$$
D'_{t,s}
=\left\{(z_1,z_2)\in\bbC^2:
  \big|\im\,z_1-\log|z_2|^2\big|< t, \big|\log|z_2|^2\big|<s\right\}
\,,
$$
and
\begin{align*}
&\|F\|^2_{L^2(d_b(D'_\beta))}=\\
  &\quad \int_{\bbR}\int_{0}^{2\pi}\left|F\Big(
    x+i(s+t),e^{\frac{s}{2}}e^{ i\theta} \Big)\right|^2 e^{\frac{s}{2}}d\theta
dx+\int_{\bbR}\int_{0}^{2\pi}\left|F\left(
    x+i(s-t),e^{\frac{s}{2}}e^{ i\theta} \right)\right|^2\ e^{\frac{s}{2}}d\theta
dx \notag \\ 
&\quad +\int_{\bbR}\int_{0}^{2\pi}\left|F\left(
    x-i(s+t),e^{-\frac{s}{2}}e^{ i\theta} \right)\right|^2 e^{-\frac{s}{2}}d\theta
dx+\int_{\bbR}\int_{0}^{2\pi}\left|F\left(
    x-i(s-t),e^{-\frac{s}{2}}e^{ i\theta} \right)\right|^2 e^{-\frac{s}{2}}d\theta
dx. \notag
\end{align*}

We notice that
\begin{equation*}
d_b(D'_\beta)=E'_1\cup E'_2\cup E'_3\cup E'_4 \,
\end{equation*}
where
\begin{align*}
&E'_1=\big\{(z_1,z_2): \im\, z_1=\beta,\log|z_2|^2=\beta-\pt\big\};\\
&E'_2\!=\!\left\{(z_1,z_2): \im\, z_1=\beta-\pi,\log|z_2|^2=\beta-\pt\right\};\\
&E'_3=\left\{(z_1,z_2)\: \im\, z_1=-\beta,\log|z_2|^2=-\beta+\pt\right\};\\
&E'_4\!=\!\left\{(z_1,z_2): \im\, z_1=-\beta+\pi,\log|z_2|^2=-\beta+\pt\right\}
\end{align*}
and each component $E'_\ell$, $\ell=1,\ldots,4$, can be identified with
$\bbR\times\bbT$. Therefore, the restriction of any function defined
on $d_b(D'_\beta)$ to any of the component $E'_\ell$ can be identified
with a function defined on $\bbR\times\bbT$. 

Denoting by $\chi'_\ell$ the characteristic function of
$E'_\ell$, $\ell=1,\ldots,4$, the Szeg\H o projection $\sP'f$ of
a function $f\in L^2(d_b(D'_\beta))$ is given by 
$$
\sP'f = \sum_{k,\ell=1}^4 \chi'_k\sP'(\chi'_\ell f) \,,
$$
where, writing $\chi'_k \sP'(\chi'_\ell \cdot) =\sP'_{k,\ell}$  we adopt
the convention that the operator $\sP'_{k,\ell}$ is acting on
functions defined on $\bbR\times\bbT$.
 
Now, by \cite[Remark 3.13]{M} it follows that 
\begin{equation}\label{T-tilde-m-k-elle}
\sP'_{k,\ell} = \cF^{-1} \big( \tilde m_{k,\ell} \cF \big) = T_{\widetilde m_{k,\ell}}\,,
\end{equation}
where 
$$
\tilde m_{k,\ell} = m_{k,\ell} (\textstyle{\frac12} -i\cdot,\cdot)
$$
and for $k,\ell\in\{1,\dots,4\}$,  $m_{k,\ell}$ are as in
Def. \ref{mult-expressions}.

Thus, the operator $\sP'$ is obtained as a sum of Fourier multiplier
operators -- and its mapping properties were
studied  in \cite{M}. 

\ms

\subsection{Relationship between $\sP'$ and $\sP$}
We now show that $\sH^2(D_\beta)$ and $\sH^2(D'_\beta)$
are isometric and, as a consequence, we obtain a transformation rule
for the Szeg\H{o} projections $\sP'$ and $\sP$.  

The domains $D'_\beta$ and $D_\beta$ are biholomorphic via the
biholomorphism 
\begin{align}\label{Biholomorphism}
\begin{split}
\vp:\ &D'_\beta\to D_\beta \qquad\qquad\qquad\qquad \vp^{-1}: D_\beta\to D'_\beta\\
&(z_1,z_2)\mapsto (e^{z_1},z_2)\qquad \qquad\quad\qquad (z_1,z_2)\mapsto (\Log(z_1 e^{-i\log|z_2|^2})+i\log|z_2|^2, z_2)
\end{split}
\end{align}
where $\Log$ denotes the principal branch of the complex logarithm. It is straightforward to see that
\begin{equation*}
 d_b(D_\beta)=\vp(d_b(D'_\beta))\cup \left\{(0,z_2): \big|\log|z_2|^2\big|=\beta-\pt\right\}.
\end{equation*}

The following proposition holds.
\begin{prop}\label{Isometry}
For $1<p<\infty$ be fixed.
 Let $\psi_p$ be
given by
\begin{equation*}
 \psi_p(z_1,z_2) := e^{-\frac ip\log|z_2|^2}(z_1
 e^{-i\log|z_2|^2})^{-\frac 1p} \,.
\end{equation*}
Let $\vp^{-1}$ be as in 
\eqref{Biholomorphism} and define the operator 
\begin{equation*}
\Lambda f:= \psi_p(f\circ\vp^{-1}) \,.
\end{equation*}
Then, $\psi_p \in \operatorname{Hol}(D'_\beta)$ and 
$$
\Lambda:
\sH^p(D'_\beta)\to \sH^p(D_\beta)
$$
is a surjective isometry.
\end{prop}

\begin{proof}
 It is easy to see that $\psi_p\in \operatorname{Hol}(D'_\beta)$ (or
see \cite[Lemma 1.2]{KPS}). 
The holomorphicity of $\Lambda f$ on $D_\beta$ follows immediately
from the holomorphicity of $\vp^{-1}$ and $\psi_p$, whereas the
equality $\|f\|_{\sH^p(D'_\beta)}=\|\Lambda f\|_{\sH^p(D_\beta)}$
follows at once from the fact that $\|f\|_{L^p(d_b(D'_{t,s}))}=\|\Lambda f\|_{L^p(d_b(D_{t,s}))}$ for every
$(t,s)\in (0,\pt)\times[0,\beta-\pt)$, as can be seen by
a simple change of variables.
\end{proof}

We now notice that this proposition implies part {\it (ii)} in
Proposition \ref{1st-prop-Hp}. 
For, let $F\in H^p(D_\beta)$ and set $f=\Lambda^{-1} F \in H^p(D'_\beta)$.
By \cite[Thm. 4.3]{M} we know that the analogous conclusion holds true
for $f$.  Using part{\it (i)}, {\it (ii)} in Proposition
\ref{1st-prop-Hp} now follows easily. \ms

Using the proposition, we obtain a transformation rule for the
Szeg\H{o} kernels and projections of the spaces $\sH^2(D_\beta)$ and
$\sH^2(D'_\beta)$ similar to the one proved by Bell in \cite[Thm.
12.3]{MR1228442} for smooth bounded domains in $\bbC$.  We leave the
elementary details to the reader.

\begin{prop}\label{transf-rule}
Let $K$ and $K'$ be the reproducing kernels of $\sH^2 (D_\beta)$ and
$\sH^2 (D'_\beta)$ respectively. Then, 
\begin{equation*} 
  K\big((z_1,z_2),(w_1,w_2)\big)
=\psi_2(z_1,z_2)\, K'\big(\vp^{-1}(z_1,z_2),\vp^{-1}(w_1,w_2)\big)\,
  \overline{\psi_2(w_1,w_2)}.
\end{equation*}
Hence, if $\sP'$ and $\sP$ are the Szeg\H{o} projections of
$\sH^2(D'_\beta)$ and $\sH^2(D_\beta)$ respectively and $f\in L^2(d_b(D_\beta))$, then 
\begin{equation*}
  \sP'(\Lambda ^{-1}f)=\Lambda ^{-1}(\sP f).
\end{equation*}
\end{prop}

Finally, we have,

\proof[End of proof of Theorem \ref{Pll'}]
From Proposition
\ref{transf-rule}  we have that
\begin{equation}\label{P-P'}
\sP f = \sum_{k,\ell=1}^4 \sP_{k,\ell}f 
=\sum_{k ,\ell=1}^4 \chi_k \sP(\chi_\ell f)
=\sum_{k,\ell=1}^4 \chi_k\psi_2 \cdot \sP' 
\big( ({\textstyle \frac{1}{\psi_2}} \chi_\ell f)\circ
\vp \big) \circ \vp^{-1}\,.
\end{equation}

Observe that
$$
\sP' \big( ({\textstyle \frac{1}{\psi_2}} \chi_\ell f)\circ
\vp \big) =
\sum_{k,m=1}^4 \chi'_k \sP' \big( \chi'_m \cdot ({\textstyle \frac{1}{\psi_2}} \chi_\ell f)\circ
\vp \big) =  \sum_{k=1}^4 \chi'_k \sP' \big( \chi'_\ell\cdot ({\textstyle \frac{1}{\psi_2}} \chi_\ell f)\circ
\vp \big)\,,
$$
since $\chi_\ell\circ\vp =\chi'_\ell$, so it follows that 
$\chi'_m \cdot ({\textstyle \frac{1}{\psi_2}} \chi_\ell f)\circ
\vp \big)$ does not vanish identically  only if $m=\ell$.
Therefore, using \eqref{T-tilde-m-k-elle}
$$
\sP' \big( ({\textstyle \frac{1}{\psi_2}} \chi_\ell f)\circ
\vp \big) 
=  \sum_{k=1}^4  \sP'_{k,\ell} \big(  ({\textstyle \frac{1}{\psi_2}} \chi_\ell f)\circ
\vp \big)  
= \sum_{k=1}^4 T_{\widetilde m_{k,\ell}}  \big(  ({\textstyle \frac{1}{\psi_2}} \chi_\ell f)\circ
\vp \big) 
\,
$$
where $T_{\widetilde m_{k,\ell}}$ denotes the Fourier multiplier
operator associated to $\widetilde m_{k,\ell}$ as in Section
\ref{M-F}. 

It is now an easy observation that 
$$
 ({\textstyle \frac{1}{\psi_2}} \chi_\ell f)\circ
\vp=\cC_2 (\chi_\ell f),
$$
hence, from \eqref{P-P'}, the results in Section \ref{M-F} and the
first part of Theorem \ref{Pll'}, we conclude that
\begin{align*}
 \sP f&=\sum_{k,\ell=1}^4 \cC_2^{-1} T_{\widetilde m_{k,\ell}} \cC_2 f
=\sum_{k,\ell=1}^4 \sT_{m_{k,\ell}}f
\end{align*}
as we wished.  This concludes the proof of the theorem.
\ms
\qed

\section*{Final Remarks}
As we mentioned, our final goal is to prove the (ir-)regularity of the
Sze\H o projection on the smooth, bounded, worm domain $\Omega_\beta$. 
In this case, 
there is no ambiguity in the definition of the (classical) space
$H^2(\Omega_\beta)$ and the associated Szeg\H{o} projection on the
topological boundary. 
On the other hand, due to the nature of the domain $D_\beta$, and of
its biholomorphic copy $D'_\beta$, it was a natural choice to define
and study the Szeg\H{o} projection on the distinguished boundary. The
drawback of this choice is that it is more complicated to transfer
information from the Hardy space $\sH^2(d_b(D_\beta))$ to
$H^2(b\Omega_\beta)$. We do 
not exclude that to fully understand the behavior of the Szeg\H{o} 
projection $P_{\Omega_\beta}$ on the topological boundary
$b\Omega_\beta$ it might be necessary to study the Hardy spaces on the
topological boundary of $D_\beta$. The Hardy spaces on the topological
boundary of $D'_\beta$ have already been studied by  the same authors
in \cite{MP}. 

Finally, we remark that regularity of the Szeg\H{o} projection, at
least in a certain setting, is equivalent to the regularity of the
Complex Green operator \cite{HPR}. Therefore, 
the (ir-)regularity of the  Szeg\H{o} projection $P_{\Omega_\beta}$ will also
provided information about the (ir-)regularity of the Complex Green
operator on $b\Omega_\beta$.   We plan to come back to these questions
in future works.

\bigskip

\bibliographystyle{amsalpha}
\bibliography{bibWormPlane-1}

\end{document}